\documentclass[11pt,a4paper]{amsart}

\usepackage[left=2.7cm,right=2.7cm,top=3.5cm,bottom=3cm]{geometry}

\usepackage[english]{babel}
\usepackage[T1]{fontenc}
\usepackage[latin1]{inputenc}
\usepackage{amssymb,amsmath,amsthm}
\usepackage{tikz-cd}
\usepackage[hidelinks]{hyperref}
\usepackage[shortlabels]{enumitem}
\usepackage{stmaryrd}
\usepackage{mdwlist}
\usepackage{url}
\usepackage{mathrsfs}

\newtheorem{definition}{Definition}[section]
\newtheorem{theorem}{Theorem}[section]
\newtheorem{proposition}{Proposition}[section]
\newtheorem{lemma}{Lemma}[section]

\newtheorem{conjecture}{Conjecture}[section]
\newtheorem{assumption}{Assumption}[section]

\newtheorem*{theorem*}{Theorem}
\newtheorem*{assumption*}{Assumption}
\newtheorem*{conjecture*}{Conjecture}

\theoremstyle{remark}

\newtheorem*{remark*}{Remark}

\numberwithin{equation}{section}

\DeclareMathOperator{\Sel}{Sel}
\DeclareMathOperator{\BK}{BK}
\DeclareMathOperator{\Gal}{Gal}
\DeclareMathOperator{\GL}{GL}
\DeclareMathOperator{\AJ}{AJ}
\DeclareMathOperator{\Tam}{Tam}
\DeclareMathOperator{\Kol}{Kol}
\DeclareMathOperator{\Iw}{Iw}
\DeclareMathOperator{\unr}{unr}
\DeclareMathOperator{\ord}{ord}
\DeclareMathOperator{\f}{f}
\DeclareMathOperator{\divv}{div}
\DeclareMathOperator{\length}{length}
\DeclareMathOperator{\Sym}{Sym}
\DeclareMathOperator{\Aut}{Aut}
\DeclareMathOperator{\Hom}{Hom}
\DeclareMathOperator{\SL}{SL}
\DeclareMathOperator{\loc}{loc}
\DeclareMathOperator{\charr}{char}
\DeclareMathOperator{\Frob}{Frob}
\DeclareMathOperator{\can}{can}
\DeclareMathOperator{\Col}{Col}
\DeclareMathOperator{\CH}{CH}
\DeclareMathOperator{\JL}{JL}
\DeclareMathOperator{\res}{res}
\DeclareMathOperator{\M}{M}
\DeclareMathOperator{\Ann}{Ann}
\DeclareMathOperator{\End}{End}
\DeclareMathOperator{\NS}{NS}
\DeclareMathOperator{\Tr}{Tr}
\DeclareMathOperator{\Nm}{Nm}
\DeclareMathOperator{\cont}{cont}
\DeclareMathOperator{\Bcris}{B_{cris}}
\DeclareMathOperator{\new}{new}
\DeclareMathOperator{\st}{st}
\DeclareMathOperator{\Reg}{Reg}
\DeclareMathOperator{\reg}{reg}
\DeclareMathOperator{\Tors}{Tors}
\DeclareMathOperator{\tors}{tors}
\DeclareMathOperator{\Fitt}{Fitt}
\DeclareMathOperator{\mot}{mot}
\DeclareMathOperator{\Heeg}{Heeg}
\DeclareMathOperator{\N}{N}
\DeclareMathOperator{\intt}{int}
\DeclareMathOperator{\sing}{sing}
\DeclareMathOperator{\Per}{Per}

\newfont{\cyr}{wncyr10 scaled 1100}
\newcommand{\shaf}{\mbox{\cyr{X}}}

\newfont{\smallcyr}{wncyr10 scaled 750}

\subjclass{11F11,14C15}
\keywords{Modular forms, Heegner cycles, Kolyvagin's conjecture, non-ordinary primes}

\begin{document}

\title{Kolyvagin's conjecture for modular forms at non-ordinary primes}
\author{Enrico Da Ronche}
\thanks{The author is partially supported by the GNSAGA group of INdAM}

\begin{abstract}
In this article we prove a version of Kolyvagin's conjecture for modular forms at non-ordinary primes. In particular, we generalize the work of Wang on a converse to a higher weight Gross--Zagier--Kolyvagin theorem in order to prove the conjecture under the hypothesis that some Selmer group has rank one. The main ingredients that we use in non-ordinary setting are the signed Selmer groups introduced by Lei, Loeffler and Zerbes. We will also use a result of Wan, i.e., the $p$-part of the Tamagawa number conjecture for non-ordinary modular forms with analytic rank zero. Starting from the rank one case we will show how to prove the full version of the conjecture.
\end{abstract}

\address{Dipartimento di Matematica, Universit\`a di Genova, Via Dodecaneso 35, 16146 Genova, Italy}
\email{enrico.daronche@edu.unige.it}

\maketitle

\tableofcontents

\section{Introduction}
In \cite{K}, Kolyvagin built an Euler system of cohomology classes attached to an elliptic curve. To be more precise, let $E/\mathbb{Q}$ be an elliptic curve of conductor $N$, let $p \nmid N$ be a rational prime and let $K$ be an imaginary quadratic field such that the primes dividing $N$ split in $K$. Kolyvagin built a system of Heegner points $y_c$ in $E(K_c)$ (where $K_c$ is the ring class field of $K$ of conductor $c$) for any integer $c \geq 1$ coprime with $N$ and he used them to build Kolyvagin classes
$$\kappa_E(M,c) \in H^1(K,E[p^M])$$
for certain integers $M \geq 1$ and $c \geq 1$.
We also denote by $\kappa_E$ the set of Kolyvagin classes. Then, Kolyvagin's conjecture for elliptic curves (more generally, for modular forms of weight $2$) states that $\kappa_E \neq \{0\}$. The conjecture for ordinary primes was proved by Zhang in \cite{Z}. It is possible to generalize the construction of Kolyvagin classes, and so the Kolyvagin's conjecture, to the case of higher weight newforms. In order to do that it is necessary to build Heegner cycles on Kuga--Sato varieties over modular curves (following the approach of Nekov\'a\v{r}, see \cite{N}) and Shimura curves (following the approach of Besser, see \cite{Bes}). In \cite{W}, Wang proves it in the ordinary setting assuming that some residual Selmer group has rank one and following part of the reasoning used by Zhang in weight $2$. We are going to mimic its approach in order to generalize its result and the Kolyvagin's conjecture for non-ordinary modular forms.

\subsection{Setting} \label{sssec:Setting}

Let $N \geq 3$ and $k \geq 2$ be integers ($k$ even) and let $K$ be an imaginary quadratic field of discriminant $D_K$ such that $(N,D_K)=1$. We write $N=N^+N^-$ where the prime factors of $N^+$ split in $K$ and the prime factors of $N^-$ are inert in $K$.
Let $f \in S_k(\Gamma_0(N))$ be a newform of level $N$, weight $k$, trivial character and $q$-expansion
$$f(q)=\displaystyle \sum_{n \geq 1} a_n(f)q^n$$
and let $p \nmid 6N$ be a rational prime which splits in $K$. We assume that $f$ is non-ordinary at $p$, i.e., $a_p(f)$ is not a $p$-adic unit. We fix an embedding $\overline{\mathbb{Q}} \hookrightarrow \overline{\mathbb{Q}}_p$ and we denote by $E$ a finite extension of $\mathbb{Q}_p(\{a_n(f)\})$, the $p$-adic Hecke field of $f$, by $\mathcal{O}$ its ring of integers, by $\varpi$ a uniformizer of $\mathcal{O}$, by $\mathbb{F}=\mathcal{O}/\varpi \mathcal{O}$ its residue field and by
$$\rho_f^*:G_\mathbb{Q} \rightarrow \GL_2(E)$$
the Galois representation attached to $f$ by Deligne (i.e., $\det \rho_f^*=\chi_p^{1-k}$ where $\chi_p$ is the $p$-adic cyclotomic character). We are going to work with the self-dual twist $\rho_f:=\rho_f^*(k/2)$. We call $V_f$ the underlying vector space of $\rho_f$ and we fix a $G_\mathbb{Q}$-stable lattice $T_f$ in it. We also set $A_f:=V_f/T_f$. Finally, for any integer $n \geq 1$ we set $T_{f,n}:=T_f/\varpi^n T_f$ and $A_{f,n}:=A_f[\varpi^n]$. We know that $T_{f,n} \cong A_{f,n}$ as $G_\mathbb{Q}$-modules. We denote by
$$\overline{\rho_f}:G_\mathbb{Q} \rightarrow \GL_2(\mathbb{F})$$
the residual representation.

\subsection{Main result}

In \autoref{sec:5} we build Kolyvagin classes
$$\kappa_f(M,c) \in H^1(K,A_{f,M})$$
for certain integers $M \geq 1$ and $c \geq 1$ and we denote by $\kappa_f$ the set of these classes. Then Kolyvagin's conjecture is formulated as in the elliptic curves case.

\begin{conjecture}
$\kappa_f \neq \{0\}$.
\end{conjecture}

We will assume our data to satisfy the following hypotheses.

\begin{assumption} \label{Ass} ~
\begin{itemize}
\item  $N^+$ is square-free;
\item  $N^-$ is square-free and it consists of an even number of primes;
\item  $p>k+1$ and $\left| (\mathbb{F}_p^\times)^{k-1} \right|>5$;
\item  $\rho_f(G_\mathbb{Q})$ contains $\SL_2(\mathbb{Z}_p)$;
\item  $\overline{\rho_f}$ is ramified at any prime $q|N^+$;
\item $\overline{\rho_f}$ is ramified at any prime $q|N^-$ such that $q^2 \equiv 1 \mod p$;
\item there exists a prime $q|N$ such that $\overline{\rho_f}$ is ramified at $q$;
\item  $\overline{\rho_f}$ is absolutely irreducible when restricted to the absolute Galois group of $\mathbb{Q}_p \left( \sqrt{ p^*} \right)$ where $p^*=(-1)^{\frac{p-1}{2}}p$;
\item there exists $q' | N^+$ such that $a_{q'}(f)=-(q')^{k/2-1}$ or there exist $q_i' | N^-$ for $i=1,2$ such that $a_{q_i'}(f)=(-1)^i (q_i')^{k/2-1}$.
\end{itemize}
\end{assumption}

The first result of our paper is the proof of Kolyvagin's conjecture (and, in particular, non-triviality of $\kappa_f(1,1)$) when the rank of a residual Selmer group defined in \autoref{sec:2} is one.

\begin{theorem}
If Assumption \ref{Ass} is satisfied and $\dim_\mathbb{F}\Sel_{\mathcal{F}(N^-)}(K,A_{f,1})=1$ then $\kappa_f(1,1)$ is not trivial.
\end{theorem}

Starting from this result we prove our main result, i.e., the conjecture in its generality (the higher rank case) by following closely the approach used by Zhang for ordinary modular forms of weight $2$ in \cite{Z} and generalized by Longo, Pati and Vigni in \cite{LPV} for higher weight ordinary modular forms.

\begin{theorem}
If Assumption \ref{Ass} is satisfied then $\kappa_f \neq \{0\}$.
\end{theorem}

\subsection{Strategy of proof}

The proof of the rank one case (given in \autoref{sec:6}) follows the main steps of the proof given by Wang in the ordinary setting in \cite[Theorem 4.2]{W}. The main difference is that we work with different Selmer structures (defined in \autoref{sec:2}) since Greenberg local conditions at $p$ used by Wang are defined only for ordinary primes. Instead, we use signed local conditions at $p$ as defined by Lei, Loeffler and Zerbes in \cite{LLZ} as they fit well in our non-ordinary context. We are going to work with them using properties proved by Hatley and Lei in \cite{HL}. We summarize the main steps of our proof.
\begin{enumerate}
\item First of all, we consider an admissible prime $\ell$ (see Definition \ref{adm}) given by Lemma \ref{eadm} and we use a level raising theorem by Wang (see Theorem \ref{lr}) to get a newform $g \in S_k(\Gamma_0(\ell N))$  satisfying certain properties.
\item We prove control theorems for our Selmer groups in order to get $\Sel_{\mathcal{F}_{\ell N^-}}(K,A_g)=0$.
\suspend{enumerate}
After these steps, Wang uses \cite[Theorem 4.5]{W}, i.e., a formula that follows from anticyclotomic Iwasawa main conjecture which is not available in the non-ordinary setting. Thus, we follow quite different steps.
\resume{enumerate}
\item We use a splitting argument and properties of signed local conditions to prove that
$$\Sel_{\BK}(\mathbb{Q},A_g)=\Sel_{\BK}(\mathbb{Q},A_{g^K})=0$$
where $g^K=g \otimes \epsilon_K$ and $\epsilon_K$ is the quadratic unramified character attached to $K$.
\item We use an important result of Wan, i.e., the $p$-part of the Tamagawa number conjecture for non-ordinary normalized eigenforms of analytic rank zero (see Theorem \ref{pTNC}) in order to prove that
$$v_\varpi \left(\frac{L(g,k/2)}{(2 \pi i)^{k/2-1} \Omega_g}\right)=\displaystyle \sum_{q| \ell N^-} t_{q}(\mathbb{Q},A_g), \: \: \: \: \: \: \: \: v_\varpi \left(\frac{L(g^K,k/2)}{(2 \pi i)^{k/2-1} \Omega_{g^K}}\right)=\displaystyle \sum_{q|\ell N^-} t_{q}(\mathbb{Q},A_{g^K}).$$
\suspend{enumerate}
With some computations we will prove that
$$v_\varpi \left(\frac{L(g/K,k/2)}{\Omega_{g,\ell N^-}} \right)=0$$
and we will conclude with the following steps.
\resume{enumerate}
\item We see that $v_\varpi(\theta_{c,N^+,\ell N^-}(\phi))=0$ using the special value formula given in Theorem \ref{svf}, where $\theta_{c,N^+,\ell N^-}(\phi)$ is the theta element introduced in Definition \ref{theta}.
\item Finally, we use the Wang reciprocity law stated in Theorem \ref{rec} and we make some computations to conclude that $\kappa_f(1,1) \neq 0$.
\end{enumerate}

\subsection{Structure of the paper}

In order to ease the reading of the paper we give a brief description of its structure. In \autoref{sec:2} we define different local conditions and Selmer groups on our Galois representations with a particular focus on signed Selmer groups. In \autoref{sec:3} we recall a level raising result by Wang and we introduce the quaternion algebras and the Shimura curves that we are interested in. We will work both in the indefinite and in the definite case. For the first one we recall how to build Heegner cycles on the Kuga--Sato varieties over our Shimura curves. For the second one we define theta elements by working in the space of $p$-adic modular forms over our quaternion algebra. The \autoref{sec:4} is devoted to recall canonical and congruence periods attached to modular forms and we use them to give the special value interpolation formula for the theta elements. In \autoref{sec:5} we define Kolyvagin classes and we state Kolyvagin's conjecture. We also give a reciprocity law by Wang which relates these classes to the theta elements. Finally, \autoref{sec:6} contains the proof of the conjecture in the rank one case and in \autoref{sec:7} we use it to prove Kolyvagin's conjecture. In \autoref{sec:8} we show how to use our results to prove the $p$-part of the Tamagawa number conjecture as stated in \cite{LV} for modular motives of analytic rank zero and one.

\subsection{Notation} We fix the following notation.
\begin{itemize}
\item For any perfect field $F$ we set $G_F:=\Gal(\overline{F}/F)$ and for any $G_F$-module $M$ we set
$$H^1(F,M):=H^1_{\cont}(G_F,M).$$
If $F$ is a number field and $v$ is a prime of $F$, we denote by $F_v$ the completion of $F$ with respect to $v$ and we set $I_v:=\Gal(\overline{F_v}/F_v^{\unr})$.
\item For any integer $n \geq 1$ we set $\mathbb{Q}_{p,n}:=\mathbb{Q}_p(\mu_{p^n})$ and $\mathbb{Q}_\infty:=\mathbb{Q}(\mu_{p^\infty})$, while $\mathbb{Q}_{\infty,p}$ is the completion of $\mathbb{Q}_\infty$ with respect to the only prime lying over $p$. We also set $\Gamma:=\Gal(\mathbb{Q}_\infty/\mathbb{Q})$ and we denote by $\Lambda_E(\Gamma):=\mathcal{O} \llbracket \Gamma \rrbracket$ the Iwasawa algebra of $\Gamma$ over $E$.
\item For any integer $c \geq 1$ we denote by $K_c$ the ring class field of $K$ of conductor $c$ and we set:
$$G_c:=\Gal(K_c/K_1), \:\:\:\:\:\:\:\: \mathcal{G}_c:=\Gal(K_c/K).$$
\end{itemize}

\subsection*{Acknowledgements}
This paper is part of my PhD thesis at Universit\`a di Genova. I would like to thank my supervisor Stefano Vigni for suggesting this project to me and for his availability. I would like to thank also Matteo Longo, Luca Mastella, Ignacio Mu\~{n}oz Jim\'enez, Maria Rosaria Pati, Beatrice Ostorero Vinci and Francesco Zerman for several helpful discussions.

\section{Selmer groups} \label{sec:2}

In this section we introduce the Selmer groups that we are going to use in our proof. The settings of \autoref{sssec:Setting} are in force.

\subsection{Bloch--Kato Selmer groups}

Let $F \in \{\mathbb{Q},K\}$ and $M \in \{V_f,T_f,A_f,T_{f,n},A_{f,n}\}$. Let $v$ be a prime of $F$ which does not lie over $p$. We denote by
$$H^1_{\unr} (F_v,M):=\ker \left( H^1(F_v,M) \rightarrow H^1(I_v,M) \right).$$
We set $H^1_{\f}(F_v,V_f):=H^1_{\unr}(F_v,V_f)$ and we define $H^1_{\f}(F_v,T_f)$ and $H^1_{\f}(F_v,A_f)$, respectively, as the preimage and the image of $H^1_{\f}(F_v,V_f)$ under the maps induced on the cohomology groups by the inclusion $T_f \hookrightarrow V_f$ and the surjection $V_f \twoheadrightarrow A_f$. We also define $H^1_{\f}(F_v,T_{f,n})$ and $H^1_{\f}(F_v,A_{f,n})$, respectively, as the image of $H^1_{\f}(F_v,T_f)$ under the map induced by $T_f \twoheadrightarrow T_{f,n}$ and the preimage of $H^1_{\f}(F_v,A_f)$ under the map induced by $A_{f,n} \hookrightarrow A_f$.
If $v$ is a prime of $F$ which lies over $p$ we set
$$H^1_{\f}(F_v,V_f):=\ker \left( H^1(F_v,V_f) \rightarrow H^1(F_v,V_f \otimes_{\mathbb{Q}_p} \Bcris) \right)$$
where $\Bcris$ is the Fontaine's crystalline ring of periods and we follow the previous constructions to define it for the other modules.
\begin{definition}
The \textbf{Bloch--Kato Selmer group} of $M$ over $F$ is
$$\Sel_{\BK}(F,M):=\ker \left( H^1(F,M) \rightarrow \displaystyle \prod_v \frac{H^1(F_v,M)}{H^1_{\f}(F_v,M)} \right).$$
\end{definition}

\subsection{Signed Selmer groups}
Since we are going to work with non-ordinary modular forms we need to define a particular local condition which works well in our context. In particular, we are going to use the signed Selmer groups introduced by Lei, Loeffler and Zerbes in \cite{LLZ}.
We consider the twist $\mathcal{T}_f:=T_f(k/2-1)$ and we set
$$H^1_{\Iw}(\mathbb{Q}_p,\mathcal{T}_f):= \varprojlim H^1(\mathbb{Q}_{p,n},\mathcal{T}_f).$$
We observe that
$$H^1(\mathbb{Q}_p,T_f)=H^1(\mathbb{Q}_p,\mathcal{T}_f(1-k/2))=H^1(\mathbb{Q}_p, \mathcal{T}_f) \otimes \chi_p^{1-k/2}$$
and so there is an obvious chain of maps
$$H^1_{\Iw}(\mathbb{Q}_p,\mathcal{T}_f) \rightarrow H^1(\mathbb{Q}_p,\mathcal{T}_f) \rightarrow H^1(\mathbb{Q}_p,T_f).$$
In \cite[$\mathsection 3$]{LLZ}, given $i=1,2$, Lei, Loeffler and Zerbes use the theory of Wach modules to define two Coleman maps
$$\Col_{f,i}:H^1_{\Iw}(\mathbb{Q}_p,\mathcal{T}_f) \rightarrow \Lambda_E(\Gamma).$$
We denote by $H^1_i(\mathbb{Q}_p,T_f)$ the image of $\ker \Col_{f,i}$ under the previous chain of maps. Since the character of $f$ is trivial, the coefficients $a_n(f)$ lie in a totally real field and so we have the local Tate pairing
$$H^1(\mathbb{Q}_p,T_f) \times H^1(\mathbb{Q}_p,A_f) \rightarrow \mathbb{Q}_p/\mathbb{Z}_p$$
and we denote by $H^1_i(\mathbb{Q}_p,A_f)$ the orthogonal complement of $H^1_i(\mathbb{Q}_p,T_f)$ under the pairing. We also denote by $H^1_i(\mathbb{Q}_{\infty , p},A_f)$ its image under the restriction map. We assume the following:
\begin{assumption}
$p>k+1$.
\end{assumption}
By \cite[Remark 2.2]{HL} we have $A_f^{G_{\mathbb{Q}_{\infty ,p}}}=0$ and so for any integer $n \geq 1$ one has
$$H^1(\mathbb{Q}_p,A_{f,n}) \cong H^1(\mathbb{Q}_p,A_f)[\varpi^n], \:\:\:\: H^1(\mathbb{Q}_{\infty ,p},A_{f,n}) \cong H^1(\mathbb{Q}_{\infty ,p},A_f)[\varpi^n]$$
by \cite[Lemma 4.1]{HL}. Thus, we can define
$$H^1_i(\mathbb{Q}_p,A_{f,n}):=H^1_i(\mathbb{Q}_p,A_f)[\varpi^n], \:\:\:\: H^1_i(\mathbb{Q}_{\infty ,p},A_{f,n}):=H^1_i(\mathbb{Q}_{\infty ,p},A_f)[\varpi^n].$$
Keeping in mind that $p$ splits in $K$, we can also define $H^1_i(K_v,M)$ in the obvious way for any $v|p$ (because $\mathbb{Q}_p=K_v$) and any $M \in \{T_f,A_f,A_{f,n} \}$.
\begin{definition}
Let $F \in \{\mathbb{Q},K\}$ and $M \in \{T_f,A_f,A_{f,n} \}$.
For $i=1,2$, the \textbf{signed Selmer group} of $M$ over $F$ is
$$\Sel_i(F,M):=\ker \left( H^1(F,M) \rightarrow \displaystyle \prod_{v \nmid p} \frac{H^1(F_v,M)}{H^1_{\f}(F_v,M)} \times \displaystyle \prod_{v|p} \frac{H^1(F_v,M)}{H^1_i(F_v,M)} \right).$$
\end{definition}

\subsection{Selmer structures}
Following \cite{H}, we give
\begin{definition}
Let $F \in \{ \mathbb{Q},K \}$ and $M \in \{V_f,T_f,T_{f,n},A_f,A_{f,n} \}$. A \textbf{Selmer structure} on $M$ over $F$ is a pair $(\Sigma_\mathcal{F},\mathcal{F})$ where
\begin{itemize}
\item $\Sigma_\mathcal{F}$ is a finite set of primes of $F$ containing the archimedean ones, the primes lying over $p$ and the places on which $M$ is ramified;
\item $\mathcal{F}$ is a collection $\{ H^1_\mathcal{F}(F_v,M) \}_v$ of subgroups of $H^1(F_v,M)$ for any place $v$ of $F$ such that $H^1_\mathcal{F}(F_v,M)=H^1_{\unr}(F_v,M)$ for any $v \notin \Sigma_\mathcal{F}$.
\end{itemize}
The Selmer group attached to the Selmer structure $(\Sigma_\mathcal{F},\mathcal{F})$ is
$$\Sel_\mathcal{F}(F,M):=\ker \left(H^1(F,M) \rightarrow \displaystyle \prod_v \frac{H^1(F_v,M)}{H^1_\mathcal{F}(F_v,M)} \right).$$
\end{definition}

We can see that the Bloch--Kato and the signed Selmer groups arise as Selmer groups attached to particular Selmer structures since $H^1_{\f}(F_v,M)=H^1_{\unr}(F_v,M)$ for any prime $v$ such that $M$ is unramified at $v$ by \cite[Lemma 3.5]{ES}.

For the proof of our main result we are going to work with a particular Selmer structure. Let $F \in \{\mathbb{Q},K \}$ and $M \in \{A_f,A_{f,n} \}$. If $q \parallel N^-$ is a rational prime and $v|q$ is a prime of $F$, by \cite{Car} we have that
\begin{equation} \label{car}
V_f \cong \begin{pmatrix} \tau_q \chi_p & c \\ 0 & \tau_q  \end{pmatrix}
\end{equation}
as $G_{\mathbb{Q}_q}$-modules where $\tau_q$ is a quadratic unramified character such that $\tau_q^2$ is trivial and $c$ is a $1$-cocycle. Then we denote by $F_v^+V_f$ the unique line of $V_f$ such that $G_{\mathbb{Q}_q}$ acts on it as $\tau_q \chi_p$. Then we also define $F_v^+A_f$ to be the image of $F_v^+V_f$ under the projection $V_f \twoheadrightarrow A_f$ and $F_v^+A_{f,n}:=F_v^+A_f \cap A_{f,n}$. We set
$$H^1_{\ord}(F_v,M):=\ker \left( H^1(F_v,M) \rightarrow H^1(F_v,M/F_v^+M) \right).$$
Given $a,b,c$ coprime integers such that $p \nmid abc$ and $c | N^-$, we define a Selmer structure $\mathcal{F}_a^b(c)$ and its associated Selmer group $\Sel_{\mathcal{F}_a^b(c)}(F,M)$ in the following way:
$$H^1_{\mathcal{F}_a^b(c)}(F_v,M):=
\begin{cases}
H^1_{\f}(F_v,M) & \text{if} \: v \nmid abcp\\
0 & \text{ if} \: v | a\\
H^1(F_v,M) & \text{ if} \: v | b\\
H^1_{\ord}(F_v,M) & \text{ if} \: v|c\\
H^1_2(F_v,M) & \text{ if} \: v|p
\end{cases}$$

\subsection{Tamagawa ideals}

Let $F \in \{\mathbb{Q},K \}$ and let $v \nmid p$ be a prime of $F$. We define the \textbf{Tamagawa number} of $A_f$ at $v$ as
$$c_v(F,A_f):=\left[H^1_{\unr}(F_v,A_f):H^1_{\f}(F_v,A_f)\right]$$
By \cite[Lemma 3.5]{ES}, we know that they are finite numbers with value $1$ if $A_f$ is unramified at $v$. We also define the \textbf{Tamagawa exponent} at $v$ as
$$t_v(F,A_f):=v_\varpi(c_v(F,A_f))$$
and the \textbf{Tamagawa ideal} at $v$ as
$$\Tam_v(F,A_f):=\left(\varpi^{t_v(F,A_f)}\right).$$
For the definition of the Tamagawa ideal of $A_f$ at $p$ over $\mathbb{Q}$ we refer to \cite[Definition 2.66]{LV} and we denote it by $\Tam_p(\mathbb{Q},A_f)$. We also define the associated Tamagawa exponent $t_p(\mathbb{Q},A_f)$ to be a non-negative integer such that $\Tam_p(\mathbb{Q},A_f)=\left(\varpi^{t_p(\mathbb{Q},A_f)}\right)$.

\section{Shimura curves} \label{sec:3}

\subsection{Level raising}

Before introducing the quaternion algebras and the Shimura curves we are going to work with we recall a crucial result due to Chida and Wang about level raising of modular forms. In order to state it we need the following definition.

\begin{definition} \label{adm}
Let $n \geq 1$ be an integer. An \textbf{admissible prime} for $f$ is a rational prime $\ell$ such that:
\begin{itemize}
\item $\ell \nmid pN$;
\item $\ell$ is inert in $K$;
\item $p \nmid \ell^2-1$;
\item $\varpi | \ell^{\frac{k}{2}}+\ell^{\frac{k-2}{2}}-\epsilon_\ell a_\ell(f)$ with $\epsilon_\ell=\pm 1$.
\end{itemize}
\end{definition}

Before stating the theorem, we need to introduce some technical hypotheses.

\begin{assumption} \label{Ass1} ~
\begin{enumerate}[start=1,label={\bfseries \emph{(A\arabic*)}}]
\item $N^-$ is square-free;
\item \label{FL1} $p>k+1$ and $\left| (\mathbb{F}_p^\times)^{k-1} \right|>5$;
\item \label{Irr1} $\overline{\rho_f}$ is absolutely irreducible when restricted to the absolute Galois group of $\mathbb{Q} \left( \sqrt{ p^*} \right)$ where $p^*=(-1)^{\frac{p-1}{2}}p$;
\item \label{Ram1} $\overline{\rho_f}$ is ramified at $q$ for any prime $q |N^-, q \equiv \pm 1 \mod p$;
\item \label{Ram2} $\overline{\rho_f}$ is ramified at $q$ for any prime $q \parallel N^+, q \equiv 1 \mod p$;
\item \label{Con1} the conductor $N'$ of the residual representation is coprime with $N/N'$;
\item \label{Ram3} there exists a prime $q \parallel N$ such that $\overline{\rho_f}$ is ramified at $q$;
\item $\overline{\rho_f}$ is irreducible when restricted to $I_q$ for any prime $q$ such that $q^2 | N$ and $q \equiv -1 \mod p$.
\end{enumerate}
\end{assumption}

\begin{theorem} \label{lr}
If Assumption \ref{Ass1} is satisfied and $\ell$ is an admissible prime for $f$ then there exists a newform $g \in S_k(\Gamma_0(\ell N))$ such that:
\begin{itemize}
\item $\overline{\rho_g} \cong \overline{\rho_f}$;
\item $a_q(g) \equiv a_q(f) \mod \varpi$ for any prime $q \neq \ell$;
\item $a_\ell(g) \equiv \epsilon_\ell \ell^{k/2-1} \mod \varpi$.
\end{itemize}
where we assume that $E$ contains the $p$-adic Hecke field of $g$.
\end{theorem}

\begin{proof}
If $N^-$ is divisible by an even number of primes, see \cite[Theorem 2.9]{W}, otherwise see \cite[Theorem 6.3]{CHI}.
\end{proof}

From now on we fix an admissible prime $\ell$ and a newform $g$ as in the previous theorem and we assume that $E$ contains the $p$-adic Hecke field of $g$. We denote by $\mathbb{T}_k(N^+,\ell N^-)$ the $\mathcal{O}$-integral Hecke algebra acting faithfully on $S_k^{\ell N^--\new}(\Gamma_0(\ell N),\mathcal{O})$, i.e., the space of cusp forms of level $\ell N$, weight $k$ and trivial character with Fourier coefficients in $\mathcal{O}$ that are new at $\ell N^-$. We fix a homomorphism
$$\theta_{g,1}: \mathbb{T}_k(N^+,\ell N^-) \twoheadrightarrow \mathcal{O}/\varpi$$
defined by sending Hecke operators to the eigenvalues attached to their action on $g$. We call $I_{g,1}$ the kernel of $\theta_{g,1}$.

\subsection{Shimura curves}

From now on we assume the following Heegner hypothesis:
\begin{assumption} \label{Heeg}
$N^-$ is square-free and it factors as a product of an even number of primes.
\end{assumption}

We are going to work with the following:
\begin{itemize}
\item $B_1$ is the indefinite quaternion algebra with discriminant $N^-$;
\item $B_2$ is the definite quaternion algebra with discriminant $\ell N^-$.
\end{itemize}

For $i=1,2$ we fix maximal orders $\mathcal{O}_i$ of $B_i$ and Eichler orders $\mathcal{R}_i$ of level $N^+$ contained in $\mathcal{O}_i$. Since the primes dividing $\ell N^-$ do not split in $K$ we have embeddings of $\mathbb{Q}$-algebras $\iota_{K,i}:K \hookrightarrow B_i$. We fix elements $J_i \in B_i^\times$ such that:
\begin{itemize}
\item $B_i=K \oplus K \cdot J_i$;
\item $J_i^2=\beta_i \in \mathbb{Q}_{<0}$;
\item $J_iz=\bar{z}J_i$ for any $z \in K$;
\item $(\beta_i)_q \in (\mathbb{Z}_q^\times)^2$ for any $q|N^+$;
\item $(\beta_i)_q \in \mathbb{Z}_q^\times$ for any $q|D_K$.
\end{itemize}
We denote by $\delta_K$ the square root of $D_K$ such that $\Im(\delta_K)>0$ and we fix square roots of $\beta_i$ in $\overline{\mathbb{Q}}$ denoted by $\sqrt{\beta_i}$. We define $\theta \in K$ by setting
$$\theta:= \begin{cases}
\frac{-D_K+\delta_K}{2} & \text{ if} \: D_K \: \text{ odd},\\
\frac{-D_K+2 \delta_K}{4} & \text{ if} \: D_K \: \text{even}.
\end{cases}$$
We set $B_{i,q}:=B_i \otimes_{\mathbb{Q}} \mathbb{Q}_q$ and we fix isomorphisms
$$\iota_{i,q}:B_{i,q} \rightarrow \M_2(\mathbb{Q}_q)$$
for any finite prime $q$ such that $B_i$ is split at $q$, asking
$$\iota_{i,q}(\theta)=\begin{pmatrix} \theta + \bar{\theta} & - \theta \bar{\theta} \\ 1 &0\end{pmatrix}, \: \: \: \:
\iota_{i,q}(J_i)=\sqrt{\beta_i}\begin{pmatrix}-1 & \theta+\bar{\theta} \\ 0 &1 \end{pmatrix} $$
for any $q|N^+p$ and
$$\iota_{i,q}(\mathcal{O}_K \otimes_\mathbb{Z} \mathbb{Z}_q) \subset \M_2(\mathbb{Z}_q)$$
otherwise.
We also set $B_{1,\infty}:=B_1 \otimes_{\mathbb{Q}} \mathbb{R}$ and we fix an isomorphism
$$\iota_{1,\infty}: B_{1, \infty} \rightarrow \M_2(\mathbb{R})$$
asking
$$\iota_{1,\infty}(\theta)=\begin{pmatrix} \theta + \bar{\theta} & - \theta \bar{\theta} \\ 1 &0\end{pmatrix}.$$
We define $(\varsigma_{c,i})_q \in B_{i,q}^\times$ for any prime $q$ and any positive integer $c$ coprime with $NpD_K$ by
$$\varsigma_{c,i,q}:= \begin{cases}
1 & \text{ if} \: q \nmid cN^+,\\
\delta_K^{-1}\begin{pmatrix} \theta &\bar{\theta} \\ 1 &1\end{pmatrix} & \text{ if} \: q | N^+,\\
\begin{pmatrix}q^{v_q(c)} &0 \\ 0 &1 \end{pmatrix} & \text{ if} \: q|c \: \text{and} \: q \: \text{splits in} \: K, \\
\begin{pmatrix}1 &q^{-v_q(c)} \\ 0 &1 \end{pmatrix}& \text{ if} \: q|c \:\text{and} \: q  \: \text{is inert in} \: K
\end{cases}$$
and we set $\varsigma_{c,i}:=(\varsigma_{c,i,q})_q \in \hat{B}_i^\times$. We define the Atkin-Lehner involution $\tau_{i,q} \in B_{i,q}^\times$ at $q$ for any prime $q$ by
$$\tau_{1,q}:= \begin{cases}
\begin{pmatrix} 0 &1 \\ -N^+ &0 \end{pmatrix} & \text{if}\: q|N^+, \\
J_1 & \text{if}\: q|N^- \infty,\\
1 & \text{if}\: q \nmid N \infty,
\end{cases} \: \: \: \:
\tau_{2,q}:= \begin{cases}
\begin{pmatrix} 0 &1 \\ -N^+ &0 \end{pmatrix} & \text{if}\: q|N^+, \\
J_2 & \text{if}\: q|\ell N^- \infty,\\
1 & \text{if}\: q \nmid \ell N \infty
\end{cases}$$
and we set $\tau_i:=(\tau_{i,q})_q \in \hat{B}_i^\times$.

Now we introduce the geometric objects attached to our quaternion algebras that we are going to work with. We start with the indefinite case and we let $d \geq 5$ be a positive integer such that $(d,Np)=1$. In what follows:
\begin{itemize}
\item $X_{N^+,N^-}$ is the compact Shimura curve over $\mathbb{Z}[1/N]$ of level $N^+$ and discriminant $N^-$ attached to $B_1$ and $\mathcal{R}_1$, i.e., the projective curve which coarsely represents the moduli problem which sends $\mathbb{Z}[1/N]$-schemes $S$ to triples $(A, \iota_A,C_A)$ where
\begin{itemize}
\item $A$ is an abelian scheme of relative dimension $2$ over $S$;
\item $\iota_A:\mathcal{O}_1 \hookrightarrow \End_S(A)$ is an embedding;
\item $C_A$ is a locally cyclic subgroup of $A[N^+]$ of order $(N^+)^2$ which is stable under the action of $\mathcal{O}_1$;
\end{itemize}
\item $X_{N^+,N^-,d}$ is the compact Shimura curve over $\mathbb{Z}[1/Nd]$ with a full level $d$-structure, i.e., the projective curve which coarsely represents the moduli problem which sends $\mathbb{Z}[1/Nd]$-schemes $S$ to $4$-tuples $(A, \iota_A,C_A,\nu_{A,d})$ where the first three elements are as before and
$$\nu_{A,d}: (\mathcal{O}_1/d)_S \rightarrow A[d] $$
is an isomorphism of $\mathcal{O}_1$-stable group schemes, where $(\mathcal{O}_1/d)_S$ is the constant $\mathcal{O}_1/d$-valued group scheme over $S$;
\item $\mathcal{A}_{N^+,N^-,d}$ is the universal abelian surface over $X_{N^+,N^-,d}$;
\item $\mathcal{W}_{N^+,N^-,d,k}$ is the Kuga--Sato variety of weight $k$ over $X_{N^+,N^-,d}$, i.e., the $\frac{k-2}{2}$-fold product of $\mathcal{A}_{N^+,N^-,d}$ over $X_{N^+,N^-,d}$.
\end{itemize}
We recall that
$$X_{N^+,N^-}(\mathbb{C})=B_1^\times \textbackslash \left( \Hom_\mathbb{R}(\mathbb{C},B_{1,\infty}) \times   \hat{B}_1^\times \right) / \hat{\mathcal{R}}_1^\times \cong B_1^\times \textbackslash \left( \mathcal{H} \times   \hat{B}_1^\times \right) / \hat{\mathcal{R}}_1^\times$$
where $\mathcal{H}:=\mathbb{C} \setminus \mathbb{R}$. We will also consider the obvious maps
$$\pi_{d,1}:X_{N^+,N^-,d}\twoheadrightarrow X_{N^+,N^-}, \:\:\:\:\:\:\:\: \pi_{d,2}: \mathcal{W}_{N^+,N^-,d,k} \twoheadrightarrow X_{N^+,N^-,d}$$ and the projectors $\epsilon_d, \epsilon_k$ defined as follows:
\begin{itemize}
\item $\epsilon_d:= \frac{1}{|G_d|}\displaystyle \sum_{g \in G_d}g$ where $G_d:=(\mathcal{O}_1/d)^\times/\{\pm 1\}$;
\item $\epsilon_k$ is defined as in \cite[Lemma 2.2]{W}.
\end{itemize}
Now, we turn to the definite case. We define a conic $\mathcal{C}$ over $\mathbb{Q}$ by setting
$$\mathcal{C}(A):=\{x \in B_2 \otimes_{\mathbb{Q}} A : \Tr(x)=\Nm(x)=0 \}/A^\times$$
for any $\mathbb{Q}$-algebra $A$. We also consider the finite set $\{\mathcal{R}'_1,...,\mathcal{R}'_h \}$ of conjugacy classes of oriented eichler orders of level $N^+$ in $B_2$ (it is in bijection with $B_2^\times \textbackslash \hat{B}_2^\times / \hat{\mathcal{R}}_2^\times$) and we set $\Gamma_j:=\mathcal{R}'_j/\{\pm 1\}$ for any $j=1,...,h$. Finally, we define the Gross curve of level $N^+$ and discriminant $\ell N^-$ attached to $B_2$ and $\mathcal{R}_2$ as
$$X_{N^+,\ell N^-}:=\displaystyle \coprod_{j=1}^h \Gamma_j \textbackslash \mathcal{C}.$$
In this case too, we have
$$X_{N^+,\ell N^-}(\mathbb{C}) \cong B_2^\times \textbackslash \left( \Hom_\mathbb{R}(\mathbb{C},B_{2,\infty}) \times   \hat{B}_2^\times \right) / \hat{\mathcal{R}}_2^\times.$$

\subsection{Indefinite case: Heegner cycles}

In this subsection we are going to recall how to build Heegner cycles over $B_1$ following the approach of Besser (see \cite{Bes}, see also \cite[$\mathsection 7.4$]{CHI}). Let $c \geq 1$ be an integer such that $(c,NpD_K)=1$. First of all we define the Heegner point of level $c$ as
$$P_c:=[\theta,\varsigma_{c,1} \tau_1] \in X_{N^+,N^-}(K_c).$$
Thanks to a well-known moduli interpretation, $P_c$ is represented by a triple $(A_c,\iota_{A_c},C_{A_c})$ where $A_c$ is a complex abelian surface with quaternionic multiplication by $\mathcal{O}_1$ and with complex multiplication by $\mathcal{O}_c$, i.e., the order of $\mathcal{O}_K$ of conductor $c$. Then we have $A_c \cong E_1 \times E_2$ where $E_1 \cong \mathbb{C}/\mathcal{O}_c$, $E_2 \cong \mathbb{C}/\mathfrak{a}$ and $\mathfrak{a}$ is a fractional ideal of $\mathcal{O}_c$ (see \cite[$\mathsection 2.2$]{EV}).
We choose an element $\tilde{P}_c \in \pi_{d,1}^{-1}(P_c)$, we set $\zeta_c:=c \delta_K \in \mathfrak{a}=\Hom(E_1,E_2)$ and we denote by $\Gamma_c$ the graph of multiplication by $\zeta_c$ in $E_1 \times E_2$. We define the class
$$Z_c:=[\Gamma_c]-[E_1 \times \{ 0 \}]+cD_K[\{0\} \times E_2] \in \NS(A_c)$$
and we choose a class $\tilde{Z}_c \in \CH^1(A_c) \otimes_\mathbb{Z} \mathbb{Z}_p$ which projects to $c^{-1}Z_c \in \NS(A_c) \otimes_\mathbb{Z} \mathbb{Z}_p$. We denote by
$$\iota_{\tilde{P}_c}:(\pi_{d,2})^{-1}(\tilde{P}_c)=A_c^{k/2-1}\hookrightarrow \mathcal{W}_{N^+,N^-,d,k}$$
the obvious inclusion and we define the cycle class
$$Y_c:=\epsilon_d \epsilon_k  (\iota_{\tilde{P}'_c})_*(\epsilon_k(\tilde{Z}_c)^{k/2-1})  \in \epsilon_d \epsilon_k \CH^{k/2}(\mathcal{W}_{N^+,N^-,d,k}/K_c)\otimes_\mathbb{Z} \mathbb{Z}_p.$$
By \cite[Lemma 2.2]{W}, we know that
$$\epsilon_k \CH^{k/2}(\mathcal{W}_{N^+,N^-,k,d} / K_c)\otimes_\mathbb{Z} \mathbb{Z}_p=\epsilon_k \CH_0^{k/2}(\mathcal{W}_{N^+,N^-,k,d} / K_c) \otimes_\mathbb{Z} \mathbb{Z}_p$$
and so we can consider $Y_c \in \epsilon_d \epsilon_k \CH_0^{k/2}(\mathcal{W}_{N^+,N^-,k,d} / K_c)\otimes_\mathbb{Z} \mathbb{Z}_p $. It is well-known (see \cite[$\mathsection 4$]{N}) that the $p$-adic cycle map induces the $p$-adic Abel-Jacobi map
$$\AJ_{k,d,c}:\epsilon_d \epsilon_k \CH_0^{k/2}(\mathcal{W}_{N^+,N^-,k,d}/ K_c)\otimes_\mathbb{Z} \mathbb{Z}_p \rightarrow H^1(K_c,T_f).$$
Finally, we set
$$y_c:=\AJ_{k,d,c}(Y_c).$$
\begin{definition}
The cycle $Y_c$ is the \textbf{geometric Heegner cycle} of conductor $c$ and $y_c$ is the \textbf{arithmetic Heegner cycle} of conductor $c$.
\end{definition}

\subsection{Definite case: quaternionic modular forms and theta elements}

In this subsection we are going to introduce some notions and elements in the definite case, i.e., we will focus on $g$ and $B_2$. We denote by $\Sym^{k-2}(\mathcal{O})$ the vector space of homogenous polynomials of degree $k-2$ with coefficients in $\mathcal{O}$ and we fix the standard basis $\{e_i\}_{-k/2+1 \leq i \leq k/2-1}$ where $e_i:=X^{k/2-1-i}Y^{k/2-1+i}$. We define a representation
$$\rho_k:\GL_2(\mathcal{O}) \rightarrow \Aut_\mathcal{O}\Sym^{k-2}(\mathcal{O})$$
by setting
$$\rho_k(g)(P):=\det(g)^{k/2-1}P(aX+cY,bX+dY)$$
where $g=\begin{pmatrix} a &b \\ c & d \end{pmatrix} \in \GL_2(\mathcal{O})$ and $P \in \Sym^{k-2}(\mathcal{O})$.

\begin{definition}
The space of \textbf{p-adic quaternionic modular forms} of weight $k$ and level $N^+$ on $B_2$ is
$$S_k^{B_2}(N^+,\ell N^-):=\left\{ \phi: \hat{B}_2^\times \rightarrow \Sym^{k-2}(\mathcal{O}) : \phi(abr)=\rho_k(r_p^{-1})\phi(b) \: \forall \: a \in B_2^\times, \: b \in \hat{B}_2^\times, \: r \in \hat{\mathcal{R}}_2^\times \right\}.$$
\end{definition}

We want to associate with $g$ a $p$-adic modular form over $B_2$ by using the Jacquet--Langlands correspondence.
\begin{theorem}
There exists an isomorphism
$$\JL:S_k^{\ell N^--\new}(\Gamma_0(\ell N),\mathcal{O}) \rightarrow S_k^{B_2}(N^+,\ell N^-).$$
\end{theorem}
\begin{proof}
See \cite[Theorem 2.2]{CL}.
\end{proof}
From now on we set $\phi:=\JL(g)$. We define a pairing
$$\langle \cdot , \cdot \rangle_k:=\Sym^{k-2}(\mathcal{O}) \times \Sym^{k-2}(\mathcal{O}) \rightarrow \mathcal{O}$$
by setting
$$\left\langle \displaystyle \sum_{i=-k/2+1}^{k/2-1} a_ie_i, \displaystyle \sum_{i=-k/2+1}^{k/2-1} b_i e_i \right\rangle_k:= \displaystyle \sum_{i=-k/2+1}^{k/2-1} a_ib_{-i} \frac{\Gamma(k/2+i)\Gamma(k/2-i)}{\Gamma(k-1)}.$$
We can also define a pairing
$$\langle \cdot , \cdot \rangle_{B_2}: S_k^{B_2}(N^+,\ell N^-) \times S_k^{B_2}(N^+,\ell N^-) \rightarrow \mathcal{O}$$
on the space of $p$-adic modular forms by setting
$$\langle \phi_1 , \phi_2 \rangle_{B_2}:=\displaystyle \sum_{b} \frac{1}{|(B_2^\times \cap b \hat{\mathcal{R}}_2^\times b^{-1} \hat{\mathbb{Q}}^\times)/\mathbb{Q}^\times|} \langle \phi_1(b),\phi_2(b \tau_2) \rangle_k$$
where $b$ runs over a set of representatives of $B_2^\times \backslash \hat{B}_2^\times / \hat{\mathcal{R}}_2^\times \hat{\mathbb{Q}}^\times$.
We can see $S_k^{B_2}(N^+,\ell N^-)$ as a $\mathbb{T}(N^+,\ell N^-)$-module by considering the Hecke action on $S_k^{\ell N^--\new}(\Gamma_0(\ell N),\mathcal{O})$ and the Jacquet--Langlands correspondence. Then, the pairing $\langle \cdot , \cdot \rangle_{B_2}$ can be seen as a pairing of $\mathbb{T}(N^+,\ell N^-)$-modules and it induces a pairing of $\mathbb{T}(N^+,\ell N^-)$-modules
$$\langle \cdot , \cdot \rangle_{B_2,g} :S_k^{B_2}(N^+,\ell N^-)/I_{g,1} \times S_k^{B_2}(N^+,\ell N^-)[I_{g,1}] \rightarrow \mathcal{O}/\varpi.$$

Let $c$ be a positive integer such that $(c,\ell NpD_K)=1$. We define the Gross point of conductor $c$ as
$$P_c:=[(\iota_{K,2} \otimes \mathbb{R},\varsigma_{c,2})] \in X_{N^+,\ell N^-}(K_c).$$
Finally, we consider the theta elements introduced by Chida and Hsieh in \cite{CH2}.

\begin{definition} \label{theta}
The \textbf{theta element} of conductor $c$ attached to a $p$-adic modular form $\phi$ over $B_2$ is
$$\theta_{c,N^+,\ell N^-}(\phi):=\displaystyle \sum_{\sigma \in \mathcal{G}_c} \left\langle \rho_k^{-1} \left( \begin{pmatrix} \sqrt{\beta_2} & - \bar{\theta} \sqrt{\beta_2} \\ -1 & \theta \end{pmatrix} \right) \cdot (-D_K)^{k/2-1}e_0, \phi(\sigma(P_c))\right\rangle_k \cdot \sigma \: \in \mathcal{O}[\mathcal{G}_c].$$
\end{definition}

\section{Periods of modular forms} \label{sec:4}

We recall briefly how to define the canonical periods introduced in \cite{V} for a newform $f \in S_k(\Gamma_0(N))$. We consider the $\mathcal{O}$-Hecke algebra $\mathbb{T}_k(N)$ acting on $S_k(\Gamma_0(N),\mathcal{O})$ and the usual eigenvalues map
$$\lambda_f:\mathbb{T}_k(N) \rightarrow \mathcal{O}.$$
We set $I_f:=\ker \lambda_f$ and $\mathfrak{m}_f:=\lambda_f^{-1}((\varpi))$. We recall the Eichler-Shimura map
$$\Per:S_k(\Gamma_0(N)) \rightarrow H^1(\Gamma_0(N),\Sym^{k-2}(\mathbb{C})), \:\:\:\: \Per(f):=\left(\gamma \mapsto \int_\tau^{\gamma(\tau)} f(z)(z^{k-1},z^{k-1}, \dots, 1)dz \right),$$
we set $\omega_f:=\Per(f)$ and we recall that thanks to the action of complex conjugation we can decompose it as $\omega_f=\omega_f^+ + \omega_f^-$. We also know that $H^1(\Gamma_0(N),\Sym^{k-2}(\mathcal{O}))_{\mathfrak{m}_f}^\pm[I_f]$ are free $\mathcal{O}$-modules of rank $1$ and we fix basis $\gamma^+,\gamma^-$. Finally we define the canonical periods $\Omega_f^\pm \in \mathbb{C}$ by $\omega_f^\pm=\Omega_f^\pm \gamma^\pm$. We are going to work with $\Omega_f:=\Omega_f^\epsilon$ where $\epsilon:=(-1)^{k/2-1}$.

Now, we are going to define congruence periods attached to $g$. We denote by $\langle g,g \rangle_{\Gamma_0(\ell N)}$ the Petersson inner product of $g$ with itself and we define the congruence ideal of $g$ as
$$\eta_{g,\ell N}:=\lambda_g(\Ann_{\mathbb{T}_k(\ell N)_{\mathfrak{m}_g}} (\ker \lambda_g)).$$
The congruence number of $g$ is $\eta_g(\ell N):=\varpi^t$ where $t$ is such that $\eta_{g,\ell N}=(\varpi^t)$.

\begin{definition}~\\
The \textbf{canonical Hida period} of $g$ is
$$\Omega_g^{\can}:=\frac{4^k \pi^k \langle g,g \rangle_{\Gamma_0(\ell N)}}{\eta_g(\ell N)}.$$
The \textbf{Gross period} of $g$ is
$$\Omega_{g,\ell N^-}:=\frac{4^k \pi^k \langle g,g \rangle_{\Gamma_0(\ell N)}}{\langle \phi,\phi \rangle_{B_2}}.$$
\end{definition}

We give a special value interpolation formula for the theta elements introduced in Definition \ref{theta} that we are going to use in our proof.

\begin{theorem} \label{svf}
Let $\chi$ be a finite order character of $\mathcal{G}_c$. Then
$$\chi \left(\theta^2_{c,N^+,\ell N^-}(\phi) \right)=(-1)^c \cdot \Gamma(k/2)^2 \cdot (-D_K)^{k-1} \cdot \delta_K^{-1} \cdot \chi(N^+) \cdot \frac{\left| \mathcal{O}_K^\times\right|^2}{8} \cdot \frac{L(g/K,\chi,k/2)}{\Omega_{g,\ell N^-}}.$$
\end{theorem}

\begin{proof}
See \cite[Theorem 3.1]{W}.
\end{proof}

We are going to use also the following formula.

\begin{theorem} \label{pc}
We have
$$\left(\frac{\Omega_g \Omega_{g^K} \eta_g(\ell N)}{ \pi^2 \langle g,g \rangle_{\Gamma_0(\ell N)}} \right)=\mathcal{O}.$$
\end{theorem}

\begin{proof}
See \cite[Theorem 3.4]{AB}.
\end{proof}

\section{Kolyvagin's conjecture} \label{sec:5}

\subsection{Kolyvagin classes}
In this subsection we are going to use Heegner cycles to define classes following a method due to Kolyvagin (see \cite{G}). We follow quite closely the exposition of \cite[$\mathsection 3.2$]{LPV}.

\begin{definition}
A rational prime $q$ is called a \textbf{Kolyvagin prime} for the data $(f,p,K)$ if
\begin{itemize}
\item $q \nmid Np$;
\item $q$ is inert in $K$;
\item $M(q):=\min \{v_\varpi (a_q(f)), v_\varpi(q+1) \} >0$.
\end{itemize}
\end{definition}

We denote by $\mathcal{P}_{\Kol}(f)$ the set of Kolyvagin primes and by $\Lambda_{\Kol}(f)$ the set of Kolyvagin integers, i.e., the square-free products of Kolyvagin primes. For any $c \in \Lambda_{\Kol}(f), c>1$, we define its Kolyvagin index $M(c)$ as the minimum between the values $M(q)$ varying $q$ among the prime divisors of $c$ (and $M(1):=\infty$). Given $q \in \mathcal{P}_{\Kol}(f)$ and $c \in \Lambda_{\Kol}(f)$ we recall that
\begin{itemize}
\item $G_q=\langle \sigma_q \rangle$ is a cyclic group of order $q+1$,
\item $G_c=\displaystyle \prod_{\ell|c}G_\ell$
\end{itemize} 
and we define the Kolyvagin derivative operators as
$$D_q:= \displaystyle \sum_{i=1}^q i \sigma_q^i \in \mathbb{Z}[G_q], \: \: \: \: D_c:= \displaystyle \prod_{\ell|c} D_\ell \in \mathbb{Z}[G_c].$$
We write $\Lambda(K_c)$ for the image of $\AJ_{k,d,c}$ and set
$$z_c:=\displaystyle \sum_{\sigma \in \mathcal{G}_c} \sigma (D_c(y_c)) \in \Lambda(K_c).$$
Furthermore, by \cite[Corollary 2.7]{LV3} and \cite[Proposition 2.8]{LV3}, there exists a Galois-equivariant injection
$$\iota_{K_c,M}:\Lambda(K_c)/\varpi^M \Lambda(K_c) \hookrightarrow H^1(K_c,A_{f,M}).$$
We also notice that the map
$$\res_{K_c/K}:H^1(K,A_{f,M}) \rightarrow H^1(K_c,A_{f,M})^{\mathcal{G}_c}$$
is an isomorphism (we just need to consider the obvious inflation-restriction exact sequence and see that, since $K_c/\mathbb{Q}$ is solvable, $H^0(K_c,A_{f,M})=0$ by \cite[Lemma 3.10]{LV3}).
For any integer $1 \leq M \leq M(c)$ we have that the class $[z_c]_M \in (\Lambda(K_c)/\varpi^M \Lambda(K_c))^{\mathcal{G}_c}$ and so we can define the \textbf{Kolyvagin class} of conductor $c$ and index $M$ attached to $f$ as
$$\kappa_f(M,c):=\res_{K_c/K}^{-1}(\iota_{K_c,M}([z_c]_M)).$$
We can define the \textbf{Kolyvagin set} and the \textbf{strict Kolyvagin set} attached to our data:
$$\kappa_f:=\{\kappa_f(M,c): c \in \Lambda_{\Kol}(f), M \leq M(c)\}, \:\:\:\: \kappa_f^{\st}:=\{\kappa_f(1,c):c \in \Lambda_{\Kol}(f) \}.$$

Finally, we can state Kolyvagin's conjecture about the existence of non-trivial Kolyvagin classes.
\begin{conjecture}
$\kappa_f^{\st} \neq \{0\}$.
\end{conjecture}
Obviously, the validity of the previous conjecture implies that also $\kappa_f \neq \{0\}$.

\subsection{Reciprocity law}

In \cite{W}, Wang proves a reciprocity law which gives a relation between the theta elements of Definition \ref{theta} and Kolyvagin classes. This result will be crucial in our proof. Thanks to \cite[(3.20)]{W} we have an isomorphism
$$S_k^{B_2}(N^+,\ell N^-)/I_{g,1} \otimes_{\mathbb{F} } \mathbb{F}[\mathcal{G}_c] \cong \displaystyle \bigoplus_{v|\ell} H^1_{\f}(K_{c,v},A_{f,1})$$
and so, thanks to \cite[Proposition 4.7]{CH}, we will consider $\loc_\ell(\sigma(\iota_{K_c,1}([y_c]_1)))$ as an element of the former for any $\sigma \in \mathcal{G}_c$ (here $\loc_\ell:H^1(K_c,A_{f,1}) \rightarrow \bigoplus_{v|\ell} H^1(K_{c,v},A_{f,1})$ is the usual localization map). We also denote by $\phi_1$ a generator of $S_k^{B_2}(N^+,\ell N^-)[I_{g,1}]$.

\begin{theorem} \label{rec}
If Assumption \ref{Ass1} holds, then there exists $u \in \mathcal{O}^\times$ such that
$$\displaystyle \sum_{\sigma \in \mathcal{G}_c} \big\langle \loc_\ell \big(\sigma (\iota_{K_c,1}([y_c]_1 ))\big) , \phi_1 \big\rangle_{B_2, g} \equiv u \cdot \theta_{c,N^+,\ell N^-}(\phi) \mod \varpi.$$
\end{theorem}

\begin{proof}
See \cite[Theorem 3.4]{W}.
\end{proof}

\section{Proof of Kolyvagin's conjecture: rank one case} \label{sec:6}

In this section we generalize the work of Wang to prove Kolyvagin's conjecture under the assumption that the rank of some Selmer group is one. We work under the following hypotheses.

\begin{assumption} \label{Ass2}~
\begin{enumerate}[start=1,label={\bfseries \emph{(H\arabic*)}},ref={\bfseries{(H\arabic*)}}]
\item \label{N+} $N^+$ is square-free;
\item \label{N-} $N^-$ is square-free and it consists of an even number of primes;
\item \label{FL} $p>k+1$ and $\left| (\mathbb{F}_p^\times)^{k-1} \right|>5$;
\item \label{BI} $\rho_f(G_\mathbb{Q})$ contains $\SL_2(\mathbb{Z}_p)$;
\item \label{RamN+} $\overline{\rho_f}$ is ramified at any prime $q|N^+$;
\item \label{RamN-} $\overline{\rho_f}$ is ramified at any prime $q|N^-$ such that $q^2 \equiv 1 \mod p$;
\item \label{RamN} there exists a prime $q|N$ such that $\overline{\rho_f}$ is ramified at $q$;
\item \label{Irr} $\overline{\rho_f}$ is absolutely irreducible when restricted to the absolute Galois group of $\mathbb{Q}_p \left( \sqrt{ p^*} \right)$ where $p^*=(-1)^{\frac{p-1}{2}}p$;
\item \label{St} there exists $q' | N^+$ such that $a_{q'}(f)=-(q')^{k/2-1}$ or there exist $q_i' | N^-$ for $i=1,2$ such that $a_{q_i'}(f)=(-1)^i (q_i')^{k/2-1}$.
\end{enumerate}
\end{assumption}

\begin{theorem}
If Assumption \ref{Ass2} is satisfied and $\dim_\mathbb{F}\Sel_{\mathcal{F}(N^-)}(K,A_{f,1})=1$ then $\kappa_f(1,1)$ is not trivial.
\end{theorem}
We split the proof into various steps identified with different lemmas.

\subsection{Raising the level}
First of all we need to choose our admissible prime $\ell$ in a proper way.

\begin{lemma} \label{eadm}
Let $c \in \Sel_{\mathcal{F}(N^-)}(K,A_{f,1}) \setminus \{ 0 \}$. There exist infinitely many admissible primes $\ell$ such that $\loc_{\ell}(c) \neq 0$ where
$$\loc_\ell:H^1(K,A_{f,1}) \rightarrow H^1(K_\ell,A_{f,1})$$
is the usual localization map.
\end{lemma}

\begin{proof}
See \cite[Theorem 6.3]{CH}.
\end{proof}

We fix an admissible prime $\ell$ for $f$ which satisfies the property stated in the previous lemma. Thanks to Theorem \ref{lr} we can find and fix a newform $g$ of level $\ell N$, weight $k$ and trivial character such that its residual Galois representation is isomorphic to that of $f$ (i.e., $\overline{\rho_f} \cong \overline{\rho_{g}}$). We assume that $E$ contains the $p$-adic Hecke fields of $g$ and $g^K:=g \otimes \epsilon_K$ where $\epsilon_K$ is the unramified quadratic character attached to $K$. Recall that $g^K$ is a newform in $S_k(\Gamma_0(\ell ND_K^2))$. Since $a_p(g) \equiv a_p(f) \mod \varpi$ we can notice that both $g$ and $g^K$ are non-ordinary at $p$.

\begin{lemma} \label{sr}
We have $\Sel_{\mathcal{F}(N^-)}(K,A_{g,1}) \cong\Sel_{\mathcal{F}(N^-)}(K,A_{f,1})$.
\end{lemma}

\begin{proof}
Since the residual representations of $f$ and $g$ are isomorphic we have $A_{g,1} \cong A_{f,1}$. The finite and ordinary local conditions only depend on the Galois module and not on the modular form, so they coincide under the identification
$$H^1(K_v,A_{g,1}) \cong H^1(K_v,A_{f,1})$$
for any $v \nmid p$. In order to prove the isomorphism of the local conditions at $p$ we see that, thanks to \cite[Lemma 4.4]{HL}, we have
$$H^1_i(\mathbb{Q}_{\infty ,p},A_{g,1}) \cong H^1_i(\mathbb{Q}_{\infty,p},A_{f,1}).$$
Furthermore,
\begin{flalign*}
&& H^1_i(\mathbb{Q}_{\infty,p},A_{f,1})^{\Gal(\mathbb{Q}_{\infty ,p}/\mathbb{Q}_p)} &\cong\left( H^1_i(\mathbb{Q}_{\infty,p},A_f)[\varpi] \right)^{\Gal(\mathbb{Q}_{\infty ,p}/\mathbb{Q}_p)}&\\
&& &\cong H^1_i(\mathbb{Q}_{\infty,p},A_f)^{\Gal(\mathbb{Q}_{\infty ,p}/\mathbb{Q}_p)}[\varpi]&\\
&& &\cong H^1_i(\mathbb{Q}_p,A_f)[\varpi] \cong H^1_i(\mathbb{Q}_p,A_{f,1})
\end{flalign*}
where the third isomorphism follows from \cite[Remark 2.5]{HL}. We can conclude that
$$H^1_i(\mathbb{Q}_p,A_{g,1}) \cong H^1_i(\mathbb{Q}_p,A_{f,1})$$
and, in particular,
$$H^1_i(K_v,A_{g,1}) \cong H^1_i(K_v,A_{f,1})$$
for any $v|p$ prime of $K$.
\end{proof}

\subsection{A control theorem on Selmer groups}

\begin{lemma}
The group $\Sel_{\mathcal{F}(\ell N^-)}(K,A_{g,1})$ is trivial.
\end{lemma}

\begin{proof}
Following \cite[Proposition 2.2.9]{H} we find that
$$\dim_\mathbb{F} \Sel_{\mathcal{F}(N^-)}(K,A_{g,1}) - \dim_\mathbb{F} \Sel_{\mathcal{F}(\ell N^-)}(K,A_{g,1})=\dim_\mathbb{F} \loc_\ell (\Sel_{\mathcal{F}(N^-)}(K,A_{g,1}))=1$$
where the second equality follows from Lemmas \eqref{eadm} and \eqref{sr}. The statement follows.
\end{proof}

\begin{lemma} \label{ct}
The isomorphism
$$\Sel_{\mathcal{F}(\ell N^-)}(K,A_{g,n})\cong\Sel_{\mathcal{F}(\ell N^-)}(K,A_g)[\varpi^n]$$
holds for any integer $n \geq 1$.
\end{lemma}

\begin{proof}
We follow the idea of \cite[Proposition 1.9(ii)]{CH}. Thanks to \ref{Irr} and since $K \subset \mathbb{Q}_p$ we have $(A_{g,n})^{G_K}=0$ and so, considering the exact sequence $$0 \longrightarrow A_{g,n} \longrightarrow A_g \overset{\varpi^n}{\longrightarrow} A_g \longrightarrow 0$$
we get
\begin{equation} \label{tors}
H^1(K,A_{g,n})\cong H^1(K,A_g)[\varpi^n].
\end{equation}
Then, since the diagrams
$$\begin{tikzcd}
H^1(K,A_{g,n}) \arrow[r, ""] \arrow[d,""]   & H^1(K,A_g) \arrow[d, ""]  \\
H^1(K_v,A_{g,n}) \arrow[r, " "]    & H^1(K_v,A_g) 
\end{tikzcd}$$
are commutative for any prime $v$ of $K$ and the diagrams
$$\begin{tikzcd}
H^1(K_v,A_{g,n}) \arrow[r, " "] \arrow[d,""]   & H^1(K_v,A_g) \arrow[d,""]\\
H^1(K_v,A_{g,n}/K_v^+A_{g,n}) \arrow[r, " "]    & H^1(K_v,A_g/K_v^+A_g) 
\end{tikzcd}$$
are commutative for any $v | q$ with $q|\ell N^-$ a rational prime, we get an injective map
$$\Sel_{\mathcal{F}(\ell N^-)}(K,A_{g,n}) \hookrightarrow \Sel_{\mathcal{F}(\ell N^-)}(K,A_g)[\varpi^n].$$
In order to prove surjectivity we just need to prove that the map
$$H^1(K_v,A_{g,n}/K_v^+A_{g,n}) \rightarrow H^1(K_v,A_g/K_v^+A_g)$$
is injective for any $v | q$ with $q|\ell N^-$ a rational prime. Since $q \parallel \ell N^-$, it follows by isomorphism \eqref{car} that $V_g$ is isomorphic to $\begin{pmatrix} \chi_p &c \\ 0 &1\end{pmatrix}$ with $c$ a $1$-cocycle when the action is restricted to $G_{K_v}$. Then
$$(A_g/K_v^+A_g)^{G_{K_v}}=A_g/K_v^+A_g \cong E/\mathcal{O}$$
and so it is $\varpi$-divisible. Then the multiplication by $\varpi^n$ is surjective on $(A_g/K_v^+A_g)^{G_{K_v}}$ and we can conclude by considering the exact sequence
$$(A_g/K_v^+ A_g)^{G_{K_v}} \rightarrow (A_g/K_v^+ A_g)^{G_{K_v}} \rightarrow H^1(K_v,A_{g,n}/K_v^+ A_{g,n}) \rightarrow H^1(K_v,A_g/K_v^+ A_g).$$
\end{proof}

Since
$$\Sel_{\mathcal{F}(\ell N^-)}(K,A_g)[\varpi] \cong \Sel_{\mathcal{F}(\ell N^-)}(K,A_{g,1})=0,$$
we find that
$$\Sel_{\mathcal{F}(\ell N^-)}(K,A_g)[\varpi^n]=0$$
for any integer $n \geq 1$. Since $A_g$ is a discrete $G_K$-module, by \cite[$\mathsection$ II, Proposition 4.4]{M} and isomorphism \eqref{tors} we get
$$H^1(K,A_g)\cong\varinjlim H^1(K,A_{g,n})\cong\varinjlim H^1(K,A_g)[\varpi^n],$$
so we can conclude that $\Sel_{\mathcal{F}(\ell N^-)}(K,A_g)=0$. In particular, $\Sel_{\mathcal{F}_{\ell N^-}}(K,A_g)=0$.

\subsection{A splitting argument}
\begin{lemma} \label{splitting}
There is a splitting
$$\Sel_{\mathcal{F}_{\ell N^-}}(K,A_g )=\Sel_{\mathcal{F}_{\ell N^-}}(\mathbb{Q},A_g ) \oplus \Sel_{\mathcal{F}_{\ell N^-}}(\mathbb{Q},A_{g^K} ).$$
\end{lemma}

\begin{proof}
We follow the idea of the proof of \cite[Proposition 6.2]{LV2}.
Since $\Sel_{\mathcal{F}_{\ell N^-}}(K,A_g)$ is an $\mathcal{O}$-module (where $\charr \mathcal{O} \neq 2$) endowed with an action of $\Gal(K/\mathbb{Q})$ we have the splitting
$$\Sel_{\mathcal{F}_{\ell N^-}}(K,A_g)= \Sel_{\mathcal{F}_{\ell N^-}}(K,A_g)^{\tau=1} \oplus \Sel_{\mathcal{F}_{\ell N^-}}(K,A_g)^{\tau=-1}$$ 
where $\tau$ is the complex conjugation. First of all, we prove that
$$\Sel_{\mathcal{F}_{\ell N^-}}(K,A_g)^{\tau=1}=\Sel_{\mathcal{F}_{\ell N^-}}(\mathbb{Q},A_g).$$
We remark that a $2$-torsion group which is also an $\mathcal{O}$-module must be trivial and it is the situation for $H^i(C_2,M)$ with $i=1,2$, $C_2$ a cyclic group of order $2$ and $M$ any $C_2$-module. Using this fact we find an isomorphism
$$H^1(\mathbb{Q},A_g) \cong H^1(K,A_g)^{\Gal(K/\mathbb{Q})}.$$
We need to prove that our Selmer groups correspond under this isomorphism focusing on the different local conditions at primes $q$ of $\mathbb{Q}$ and $v|q$ of $K$:
\begin{itemize}
\item if $q|\ell N^-$, we just need to consider the commutative diagram
$$\begin{tikzcd}
H^1(\mathbb{Q},A_g) \arrow[d,""] \arrow[r,""] & H^1(K,A_g)^{\Gal(K/\mathbb{Q})} \arrow[d,""] \\
H^1(\mathbb{Q}_{q},A_g) \arrow[r,""] & H^1(K_v,A_g)^{\Gal(K_v/\mathbb{Q}_{q})} 
\end{tikzcd}$$
where all the horizontal arrows are isomorphisms by the previous considerations;
\item if $q \nmid \ell pN^-$, we consider the commutative diagram
$$\begin{tikzcd}
H^1(\mathbb{Q},V_g) \arrow[rr,""] \arrow[dd,""] \arrow[rd,""] & &H^1(K,V_g)^{\Gal(K/\mathbb{Q})} \arrow[dd,""] \arrow[dr,""] & \\
 &H^1(\mathbb{Q},A_g) \arrow[crossing over,rr,""]  & & H^1(K,A_g)^{\Gal(K/\mathbb{Q})} \arrow[dd,""] \\
H^1(\mathbb{Q}_q,V_g)\arrow[dd,""] \arrow[dr,""] \arrow[rr,""] & &H^1(K_v, V_g)^{\Gal(K_v/\mathbb{Q}_q)} \arrow[dd,""] \arrow[dr,""] & \\
 & H^1(\mathbb{Q}_q,A_g) \arrow[from=uu,crossing over] \arrow[crossing over,rr,""] & & H^1(K_v,A_g)^{\Gal(K_v/\mathbb{Q}_q)}\\
 H^1(I_{q},V_g) \arrow[rr,""] & & H^1(I_v,V_g)^{\frac{I_{q}}{I_v}} &
\end{tikzcd}$$
where all the horizontal arrows are isomorphisms by the previous considerations;
\item if $q=p$, it is obvious by definition.
\end{itemize}
With the same considerations we can prove that
$$\Sel_{\mathcal{F}_{\ell N^-}}(K,A_g)^{\tau=-1}=\Sel_{\mathcal{F}_{\ell N^-}}(\mathbb{Q},A_{g^K})$$
since
$$\Sel_{\mathcal{F}_{\ell N^-}}(K,A_g)^{\tau=-1}=\Sel_{\mathcal{F}_{\ell N^-}}(K,A_g \otimes \epsilon_K)^{\tau=1}$$
and $A_g \otimes \epsilon_K \cong A_{g^K}$.
\end{proof}

\begin{lemma}
We have
$$\Sel_{\BK}(\mathbb{Q},A_g)=\Sel_{\mathcal{F}_{\ell N^-}}(\mathbb{Q},A_g)$$
and
$$\Sel_{\BK}(\mathbb{Q},A_{g^K})=\Sel_{\mathcal{F}_{\ell N^-}}(\mathbb{Q},A_{g^K}).$$
\end{lemma}

\begin{proof}
We prove the lemma only for $A_g$ since the result for $A_{g^K}$ can be proved in the same way. Thanks to \cite[Proposition 2.14]{HL} we know that
$$\Sel_{\BK}(\mathbb{Q},A_g)=\Sel_2(\mathbb{Q},A_g),$$
so we just need to prove that
$$\Sel_2(\mathbb{Q},A_g)=\Sel_{\mathcal{F}_{\ell N^-}}(\mathbb{Q},A_g).$$
The local conditions of these Selmer groups coincide for all the primes $q$ of $\mathbb{Q}$ which do not divide $\ell N^-$. In order to conclude we prove that $H^1_{\f}(\mathbb{Q}_q,A_g)=0$ for any $q | \ell N^-$. We follow the proof of \cite[Lemma 2.54]{LV}. First of all we recall that, since $q \parallel \ell N^-$, the isomorphism \eqref{car} tells us that the restriction of $V_g$ to $G_{\mathbb{Q}_q}$ is isomorphic to $\begin{pmatrix} \tau_q \chi_p & c \\ 0 & \tau_q\end{pmatrix}$ and so the restriction to $I_q$ is isomorphic to $\begin{pmatrix} 1 & c \\ 0 & 1\end{pmatrix}$. Furthermore, we know by \cite[Theorem 3.26]{HB} that $V_g$ is ramified at $q$ and so, in particular, $c$ is not trivial when restricted to $I_{q}$ and so $V_g^{I_q} \cong E(1)$. Then
$$H^1_f(\mathbb{Q}_q,V_g) \cong V_g^{I_q}/(\Frob_q-1)V_g^{I_q} =0$$
since $\Frob_q-1$ acts isomorphically on $V_g^{I_q}$ ($\chi_p(\Frob_q)=q$). Finally $H^1_f(\mathbb{Q}_{q},A_g)=0$.
\end{proof}

Finally, we have that both $\Sel_{\BK}(\mathbb{Q},A_g)$ and $\Sel_{\BK}(\mathbb{Q},A_{g^K})$ are trivial.

\subsection{Non-ordinary $p$-part of the Tamagawa number conjecture in rank 0}
\begin{lemma}
For any $q | N^+$ we have $c_q(\mathbb{Q},A_g)=c_q(\mathbb{Q},A_{g^K})=1$. For any $q | D_K$ we have $c_q(\mathbb{Q},A_{g^K})=1$.
\end{lemma}

\begin{proof}
We follow the proof of \cite[Proposition 4.37]{LV}.
We fix $q|N^+$. We know that $A_g \cong A_{g^K}$ as $G_K$-modules and, since $q$ splits in $K$, as $G_{\mathbb{Q}_q}$-modules. Then $c_q(\mathbb{Q},A_g)=c_q(\mathbb{Q},A_{g^K})$. By \ref{N+} the isomorphism \eqref{car} (which holds for any $q \parallel N$) tells us that $V_g$ is isomorphic to $ \begin{pmatrix} \tau_q \chi_p &c \\ 0 & \tau_q \end{pmatrix}$ when the action is restricted to $G_{\mathbb{Q}_q}$ and it is isomorphic to $ \begin{pmatrix} 1 & c \\ 0 &1 \end{pmatrix}$ when the action is restricted to $I_q$. Then, since $c$ is non-trivial by \ref{RamN+} and since $A_g \cong (E/\mathcal{O})^2$ we have
$$A_g^{I_q} \cong E/\mathcal{O} \oplus \varpi^{-t} \mathcal{O}/\mathcal{O}$$
for an integer $t \geq 0$. By contradiction, we assume that $t>0$. Then
$$A_{g,1}^{I_q}=A_g^{I_q}[\varpi] \cong E/\mathcal{O}[\varpi] \oplus (\varpi^{-t}\mathcal{O}/\mathcal{O})[\varpi]=E/\mathcal{O}[\varpi] \oplus \varpi^{-1}\mathcal{O}/\mathcal{O}=(E/\mathcal{O}[\varpi])^2=A_{g,1}$$
but it contradicts \ref{RamN+}. Finally, $A_g^{I_q} \cong E/\mathcal{O}$ is divisible and so $c_q(\mathbb{Q},A_g)=1$ by \cite[Lemma 3.5]{ES}.
Now, fix $q | D_K$ and $v|q$ a prime of $K$. Using the commutativity of the diagram
$$\begin{tikzcd}
H^1(\mathbb{Q}_q,V_g) \arrow[dd,""] \arrow[rr,""] \arrow[rd,""] & &H^1(K_v,V_g)^{\Gal(K_v/\mathbb{Q}_q)} \arrow[dd,""] \arrow[rd,""] &\\
& H^1(\mathbb{Q}_q,A_g) \arrow[crossing over,rr,""] & & H^1(K_v,A_g)^{\Gal(K_v/\mathbb{Q}_q)} \arrow[dd,""] \\
H^1(I_q,V_g) \arrow[rd,""] \arrow[rr,""]& &H^1(I_v,V_g)^{I_q/I_v} \arrow[rd,""] & \\
& H^1(I_q,A_g) \arrow[from=uu,crossing over] \arrow[rr,""]& &H^1(I_v,A_g)^{I_q/I_v}
\end{tikzcd}$$
where all the horizontal arrows are isomorphisms we can find two splittings
$$H^1_{\unr}(K_{v},A_g) \cong H^1_{\unr}(\mathbb{Q}_q,A_g) \oplus H^1_{\unr}(\mathbb{Q}_q,A_{g^K}),$$
$$H^1_{\f}(K_{v},A_g) \cong H^1_{\f}(\mathbb{Q}_q,A_g) \oplus H^1_{\f}(\mathbb{Q}_q,A_{g^K})$$
with arguments similar to the ones in the proof of lemma \ref{splitting}. It follows that
\begin{equation} \label{splitting2}
\frac{H^1_{\unr}(K_{v},A_g) }{H^1_{\f}(K_{v},A_g)} \cong \frac{H^1_{\unr}(\mathbb{Q}_q,A_g)}{H^1_{\f}(\mathbb{Q}_q,A_g)} \oplus \frac{H^1_{\unr}(\mathbb{Q}_q,A_{g^K})}{H^1_{\f}(\mathbb{Q}_q,A_{g^K})}.
\end{equation}
We know that $A_g$ is unramified at $q$ (because $(N,D_K)=1$) and so
$$[H^1_{\unr}(K_{v},A_g):H^1_{\f}(K_{v},A_g)]=[H^1_{\unr}(\mathbb{Q}_q,A_g):H^1_{\f}(\mathbb{Q}_q,A_g)]=1.$$
Finally
$$c_q(\mathbb{Q},A_{g^K})=[H^1_{\unr}(\mathbb{Q}_q,A_{g^K}):H^1_{\f}(\mathbb{Q}_q,A_{g^K})]=1.$$

\end{proof}

Now, we are going to use the following crucial result of Wan, i.e., the validity of the $p$-part of the Tamagawa number conjecture for non-ordinary modular forms of analytic rank zero under some technical conditions.

\begin{theorem} \label{pTNC}
Let $f \in S_k(\Gamma_0(N))$ be a normalized eigenform and let $p \nmid N$ be a rational prime. Assume that $2|k$, $k<p$, $T$ is residually irreducible over $G_{\mathbb{Q}_p}$, $\SL_2(\mathbb{Z}_p) \subset \rho_f(G_\mathbb{Q})$ and there exists a prime $q' \parallel N$ such that $a_{q'}(f)=-(q')^{k/2-1}$.
If $L(f,k/2) \neq 0$ then
$$v_\varpi \left(\frac{L(f,k/2)}{(2\pi i)^{k/2-1} \Omega_f} \right) =\length_\mathcal{O} \Sel_{\BK}(\mathbb{Q},A_f) + \displaystyle \sum_{q|Np} t_q(\mathbb{Q},A_f).$$
\end{theorem}

\begin{proof}
See \cite[Corollary 3.35]{WAN}.
\end{proof}

We observe that $g$ and $g^K$ satisfy the assumptions of the previous theorem. Indeed:
\begin{itemize}
\item $\overline{\rho_g} \cong \overline{\rho_f}$ is irreducible over $G_{\mathbb{Q}_p}$ by \ref{Irr};
\item $\overline{\rho_{g^K}}$ is irreducible over $G_{\mathbb{Q}_p}$ because $p$ splits in $K$ and $A_g \cong A_{g^K}$ as $G_K$-representations (and so as $G_{\mathbb{Q}_p}$-representations);
\item $\SL_2(\mathbb{Z}_p) \subset \rho_g(G_\mathbb{Q})$ by \ref{BI}, surjectivity of the reduction map $\SL_2(\mathbb{Z}_p) \rightarrow \SL_2(\mathbb{Z}_p/p\mathbb{Z}_p)$ and \cite[Theorem 2.1]{RIB};
\item $\SL_2(\mathbb{Z}_p) \subset \rho_{g^K}(G_\mathbb{Q})$ because by \cite[Lemma 6.1]{CH} we have $-1 \in \overline{\rho_{g^K}}(G_{\mathbb{Q}})$ and it implies $\SL_2(\mathbb{Z}_p/p \mathbb{Z}_p) \subset \overline{\rho_{g^K}}(G_{\mathbb{Q}})$;
\item by \ref{St} and Theorem \ref{lr} there exists $q' \parallel N$ such that $a_{q'}(g)=-(q')^{k/2-1}$;
\item by \ref{St}, Theorem \ref{lr} and the equality $a_n(g^K)=a_n(g) \cdot \epsilon_K(n)$, there exists $q' \parallel N$ such that $a_{q'}(g^K)=-(q')^{k/2-1}$;
\item since $\Sel_{\BK}(\mathbb{Q},A_g)=\Sel_{\BK}(\mathbb{Q},A_{g^K})=0$ we have $L(g,k/2)\neq 0$ and $L(g^K,k/2) \neq 0$ by \cite[Corollary 3.34]{WAN}.
\end{itemize}

Since $p \nmid \ell ND_K$ we have that $V_g$ and $V_{g^K}$ are crystalline and so
$$\Tam_p(\mathbb{Q},A_g)=\Tam_p(\mathbb{Q},A_{g^K})=\mathcal{O}$$
by \cite[$\mathsection 4$]{BB}. Putting the previous results together we find that
$$v_\varpi \left(\frac{L(g,k/2)}{(2 \pi i)^{k/2-1} \Omega_g}\right)=\displaystyle \sum_{q|\ell N^-} t_q(\mathbb{Q},A_g)$$
and
$$v_\varpi \left(\frac{L(g^K,k/2)}{(2 \pi i)^{k/2-1} \Omega_{g^K}}\right)=\displaystyle \sum_{q|\ell N^-} t_q(\mathbb{Q},A_{g^K}).$$
Summing them and observing that isomorphism \ref{splitting2} holds also for $q|\ell N^-$ we get
$$v_\varpi \left( \frac{L(g/K,k/2)}{(2 \pi i)^{k-2}\Omega_g \Omega_{g^K}} \right)=\displaystyle \sum_{v|\ell N^-} t_v(K,A_g).$$
Since $i^k \in \mathcal{O}^\times$ and $2 \in \mathcal{O}^\times$, we use Theorem \ref{pc} to get
$$v_\varpi \left( \frac{L(g/K,k/2) \eta_g(\ell N)}{4^{k-1} \pi^k \langle g,g \rangle_{\Gamma_0(\ell N)}} \right)=\displaystyle \sum_{v|\ell N^-} t_v(K,A_g).$$
In particular, one has
$$v_\varpi \left( \frac{L(g/K,k/2)}{\Omega_g^{\can}} \right)=\displaystyle \sum_{v|\ell N^-} t_v(K,A_g)$$
and, since
$$v_\varpi \left(\frac{\Omega_{g,\ell N^-}}{\Omega_g^{\can}} \right) = \displaystyle \sum_{v|\ell N^-} t_v(K,A_g)$$
by \cite[Proposition 3.2]{W}, we can conclude that
$$v_\varpi \left(\frac{L(g/K,k/2)}{\Omega_{g,\ell N^-}} \right)=0.$$

\subsection{End of the proof}
If we assume $\chi$ to be the trivial character in Theorem \ref{svf} and set $c=1$ we have
$$\theta^2_{1,N^+,\ell N^-}(\phi)=(-1) \cdot \Gamma(k/2)^2 \cdot (-D_K)^{k-1} \cdot \sqrt{D_K}^{-1} \cdot \frac{\left| \mathcal{O}_K^\times\right|^2}{8} \cdot \frac{L(g/K,k/2)}{\Omega_{g,\ell N^-}}.$$
By \ref{FL} we have that $v_\varpi(\Gamma(k/2))=0$ and since $p$ splits in $K$ we have $v_\varpi(D_K)=0$. Since $p \nmid 6$ we also have $v_\varpi(|\mathcal{O}^\times|)=0$ and so $v_\varpi(\theta_{1,N^+,\ell N^-}(\phi))=0$. If we consider Theorem \ref{rec} with $c=1$ we get
$$\displaystyle \sum_{\sigma \in \mathcal{G}_1} \big\langle \loc_\ell \big( \sigma(\iota_{K_1,1}([y_1]_1)) \big) , \phi_1 \big\rangle_{B_2, g} \equiv u \theta_{1,N^+,\ell N^-}(\phi) \mod \varpi$$
with $u$ a unit in $\mathcal{O}$. Since the latter is a unit we have
$$\displaystyle \sum_{\sigma \in \mathcal{G}_1} \big\langle \loc_\ell \big(\sigma(\iota_{K_1,1}([y_1]_1)) \big) , \phi_1 \big\rangle_{B_2, g} \neq 0,$$
so, in particular,
$$\displaystyle \sum_{\sigma \in \mathcal{G}_1} \sigma \big(\iota_{K_1,1}([y_1]_1) \big) \neq 0.$$
Since $\iota_{K_1,1}$ is Galois-equivariant we have
\begin{flalign*}
&& \displaystyle \sum_{\sigma \in \mathcal{G}_1} \sigma \big(\iota_{K_1,1}([y_1]_1) \big)&=\displaystyle \sum_{\sigma \in \mathcal{G}_1} \iota_{K_1,1}\big(\sigma([y_1]_1)\big) =\iota_{K_1,1}\left(\displaystyle \sum_{\sigma \in \mathcal{G}_1} \sigma([y_1]_1)\right)&\\
&& &=\iota_{K_1,1}\left( \left[ \displaystyle \sum_{\sigma \in \mathcal{G}_1} \sigma(y_1)\right]_1 \right)=\iota_{K_1,1}([z_1]_1) \neq 0.
\end{flalign*}
Finally, $\kappa_f(1,1)=\res^{-1} \big( \iota_{K_1,1}([z_1]_1) \big) \neq 0$ and the proof is finished.

\section{Proof of Kolyvagin's conjecture: higher rank case} \label{sec:7}

In this section we are going to show how to remove the assumption on the rank of the residual Selmer group. In particular we are going to prove the following.

\begin{theorem} \label{Kconj}
If Assumption \ref{Ass2} is satisfied then $\kappa_f^{\st} \neq 0$.
\end{theorem}

The proof is essentially the same as the one given by Zhang for ordinary modular forms of weight $2$ in \cite{Z} and generalized to higher weight by Longo, Pati and Vigni in \cite{LPV}. In order to ease the reading we recall the main steps of the proof for our settings and our notations. Assumption \ref{Ass2} will be in force for the rest of the section. For any prime $v$ of $K$ we set
$$H^1_{\sing}(K_v,A_{f,1}):=H^1(I_v,A_{f,1})^{G_v/I_v}.$$
For any $* \in \{\emptyset, \f, \sing ,\ord, 2 \}$ and any meaningful rational prime $q$ we are going to denote by
$$H^1_*(K_q,A_{f,1}):=\displaystyle \bigoplus_{v|q} H^1_*(K_v,A_{f,1})$$
and by
$$\loc_q:H^1(K,A_{f,1}) \rightarrow H^1(K_q,A_{f,1})$$
the obvious localization map. First of all we need the following definition.

\begin{definition}
For any square-free integer $n \geq 1$ we call $\nu(n)$ the number of primes dividing $n$.
The \textbf{vanishing order} of $\kappa_f^{\st}$ is
$$\nu_f:= \begin{cases}
\min \{ \nu(c) : c \in \Lambda_{\Kol}(f) ,\kappa_f(1,c) \neq 0\} & \text{if} \: \kappa_f^{\st}\neq 0,\\
\infty & \text{if} \: \kappa_f^{\st}=0.
\end{cases}$$
A rational prime $q \nmid NpD_K$ is a \textbf{base point} for $\kappa_f^{\st}$ if
$$\loc_q(\kappa_f(1,c))=0$$
for any $c \in \Lambda_{\Kol}(f)$.
We denote by $\mathcal{B}_f$ the \textbf{base locus} of $f$, i.e., the set of base points for $\kappa_f^{\st}$. We also set
$$\Sel_{\mathcal{F}(N^-),\mathcal{B}_f}(K,A_{f,1}):=\ker \left(H^1(K,A_{f,1}) \rightarrow \displaystyle \prod_{q \notin \mathcal{B}_f, \: v | q} \frac{H^1(K_v,A_{f,1})}{H^1_{\mathcal{F}(N^-)}(K_v,A_{f,1})} \right).$$
\end{definition}

\subsection{Preliminary lemmas}

\begin{lemma} \label{lli}
Let $x_1,x_2$ be two linearly independent elements in $H^1(K,A_{f,1})$. Then there exist infinitely many $q \in \mathcal{P}_{\Kol}(f)$ such that
$$\loc_q(x_1) \neq 0, \:\:\:\:\:\:\:\: \loc_q(x_2) \neq 0.$$
\end{lemma}

\begin{proof}
See \cite[Lemma 3.7]{LPV}.
\end{proof}

\begin{lemma}
For any $c \in \Lambda_{\Kol}(f)$ we have $\kappa_f(1,c) \in H^1(K,A_{f,1})^{ \epsilon(f/\mathbb{Q}) \cdot (-1)^{\nu(c)}}$.
\end{lemma}

\begin{proof}
See \cite[Proposition 3.14]{LV3}.
\end{proof}

\begin{lemma} \label{sfi}
Let $q \in \mathcal{P}_{\Kol}(f)$ and $c \in \Lambda_{\Kol}(f)$ such that $q \nmid c$. Then
$$\loc_q(\kappa_f(1,cq))=0 \Leftrightarrow \loc_q(\kappa_f(1,c))=0.$$
\end{lemma}

\begin{proof}
See \cite[Lemma 3.10]{LPV}.
\end{proof}

\begin{lemma} \label{lkc}
For any rational prime $q$ and any $c \in \Lambda_{\Kol}(f)$ we have
$$\loc_q(\kappa_f(1,c)) \in \begin{cases}
H^1_{\f}(K_q,A_{f,1}) & \text{if} \: q \nmid cpN^-, \\
H^1_{\sing}(K_q,A_{f,1}) & \text{if} \: q | c ,\\
H^1_{\ord}(K_q,A_{f,1}) & \text{if} \: q | N^-,\\
H^1_2(K_q,A_{f,1}) & \text{if} \: q=p.
\end{cases}$$
\end{lemma}

\begin{proof}
For primes $q \nmid N^-$ see the proof of \cite[Theorem 3.1]{LMZ} and \cite[Proposition 3.17]{LV3} and for $q=p$ notice that we can use an argument similar to the one in the proof of Lemma \ref{ct} to see that $H^1_{\f}(K_p,A_{f,1})=H^1_2(K_p,A_{f,1})$ by \cite[Proposition 2.14]{HL}. For primes $q | N^-$ see the arguments in \cite[page 2332]{W}.
\end{proof}

\begin{lemma} \label{ltp}
The local Tate pairing induces perfect pairings among  $1$-dimensional $\mathbb{F}$-vector spaces
$$\langle \cdot, \cdot \rangle_q:H^1_{\f}(K_q,A_{f,1})^{\pm} \times H^1_{\sing}(K_q,A_{f,1})^{\pm} \rightarrow \mathbb{F}$$
for any $q \in \mathcal{P}_{\Kol}(f)$.
Furthermore, $H^1_{\f}(K_q,A_{f,1})$ is isotropic for any $q \nmid N^-p$, $H^1_{\sing}(K_q,A_{f,1})$ is isotropic for any $q \nmid Np$, $H^1_{\ord}(K_q,A_{f,1})$ is isotropic for any $q | N^-$ and $H^1_{2}(K_p,A_{f,1})$ is isotropic.
We also have
$$\displaystyle \sum_{q} \langle x,y \rangle_q =0$$
for any $x,y \in H^1(K,A_{f,1})$, where the sum runs over rational primes.
\end{lemma}

\begin{proof}
The first assertion follows from \cite[$\mathsection 5.3$]{EV}. The second assertion away from $p$ follows from \cite[Proposition 3.8]{BK} and \cite[$\mathsection 2.6.4$]{LPV}. The isotropy of $H^1_2(K_p,A_{f,1})$ follows from the equality $H^1_2(K_p,A_{f,1})=H^1_{\f}(K_p,A_{f,1})$ and \cite[Proposition 3.8]{BK}.
For the last assertion see \cite[Proposition 2.2]{Bes2}.
\end{proof}

\begin{lemma} \label{kl}
The equality
$$\Sel^{\pm}_{\mathcal{F}(\ell N^-)}(K,A_{g,1})=\ker \left( \loc_\ell: \Sel^{\pm}_{\mathcal{F}(N^-)}(K,A_{g,1}) \rightarrow H^1_{\f}(K_\ell,A_{g,1}) \right)$$
holds for any admissible prime $\ell$ and any newform $g \in S_k(\Gamma_0(\ell N))$ as in Theorem \ref{lr}. Furthermore, $\dim_\mathbb{F}H^1_{\f}(K_\ell,A_{g,1})=1$.
\end{lemma}

\begin{proof}
It follows from the proof of \cite[Proposition 2.2.9]{H}.
\end{proof}

\subsection{Proof of Kolyvagin's conjecture}

The proof of Kolyvagin's conjecture in the higher rank case is based on the two following preliminary results, i.e., the triangulation of Selmer groups and the parity conjecture.

\begin{proposition} \label{triangulation}
Assume $\kappa_f^{\st} \neq 0$ and set $\epsilon_{\nu_f}:=\epsilon(f/\mathbb{Q})\cdot (-1)^{\nu_f}$. Then:
\begin{itemize}
\item $\Sel^{\epsilon_{\nu_f}}_{\mathcal{F}(N^-)}(K,A_{f,1})=\Sel^{\epsilon_{\nu_f}}_{\mathcal{F}(N^-),\mathcal{B}_f}(K,A_{f,1})$ is a $(\nu_f +1)$-dimensional $\mathbb{F}$-vector space;
\item  $\dim_\mathbb{F} \Sel^{-\epsilon_{\nu_f}}_{\mathcal{F}(N^-),\mathcal{B}_f}(K,A_{f,1}) \leq \nu_f$.
\end{itemize}
\end{proposition}

\begin{proof}
We follow the idea of \cite[Proposition 3.14]{LPV}. In order to ease the notation let us set $\nu:=\nu_f$. First of all we prove by induction that there exists a collection $\{ q_i\}_{1 \leq i \leq 2\nu+1}$ of Kolyvagin primes such that, setting $c_i:=\displaystyle \prod_{k=i}^{\nu+i-1}q_k$, the matrix
\begin{equation} \label{mat}
( \loc_{q_{\nu+j}}(\kappa_f(1,c_i)))_{1 \leq i,j \leq \nu+1}
\end{equation}
is upper-triangular with non-zero elements on the diagonal. We notice that the definition of the vanishing order $\nu$ tells us that there exist Kolyvagin primes $q_1,\dots ,q_\nu$ such that $\kappa_f(1,c_1) \neq 0$.

Given $0 \leq j < \nu$ we assume we built a collection $\{q_1, \dots q_{\nu+j} \}$ of Kolyvagin primes which satisfy the given properties and such that $\kappa_f(1,c_{j+1}) \neq 0$. We fix a non-zero element $x \in H^1(K,A_{f,1})^{- \epsilon_\nu}$ such that
\begin{itemize}
\item $\loc_q(x) \in H^1_{\f}(K_q,A_{f,1})$ for any $q \nmid q_{j+1} \dots q_{\nu+j} p N^-$;
\item $\loc_q(x) \in H^1_{\sing}(K_q,A_{f,1})$ for any $q| q_{j+2} \dots q_{\nu+j}$;
\item $\loc_q(x) \in H^1_{\ord}(K_q,A_{f,1})$ for any $q|N^-$;
\item $\loc_p(x) \in H^1_2(K_p,A_{f,1})$.
\end{itemize}
It exists by the argument in the proof of \cite[Proposition 7.8]{Mas} and we notice that $x$ and $\kappa_f(1,c_{j+1})$ are linearly independent. Then by Lemma \ref{lli} we can find $q_{\nu+j+1} \in \mathcal{P}_{\Kol}(f)$ such that
$$\loc_{q_{\nu+j+1}}(\kappa_f(1,c_{j+1})) \neq 0, \:\:\:\:\:\:\:\: \loc_{q_{\nu+j+1}}(x) \neq 0.$$
Now we have
$$0=\langle x,\kappa_f(1,c_{j+1} q_{\nu+j+1}) \rangle=\displaystyle \sum_{q | q_{j+1}q_{\nu+j+1}} \langle x,\kappa_f(1,c_{j+1} q_{\nu+j+1}) \rangle_q$$
by Lemma \ref{lkc} and Lemma \ref{ltp} and, since $\langle x,\kappa_f(1,c_{j+1} q_{\nu+j+1}) \rangle_{q_{\nu+j+1}} \neq 0$ by Lemma \ref{ltp}, we find that $\loc_{q_{j+1}}(\kappa_f(1,c_{j+1} q_{\nu+j+1})) \neq 0$. By Lemma \ref{sfi} it implies that $\loc_{q_{j+1}}(\kappa_f(1,c_{j+2})) \neq 0$, so $\kappa_f(1,c_{j+2})  \neq 0$ and the inductive step is finished. Finally, again by Lemma \ref{lli}, we can find a Kolyvagin prime $q_{2 \nu +1}$ such that $\loc_{q_{2 \nu +1}}(\kappa_f(1,c_{\nu +1})) \neq 0$. Then the classes $\{\kappa_f(1,c_i)\}_{1 \leq i \leq \nu +1}$ are linearly independent and we can observe that they lie in $\Sel^{\epsilon_\nu}_{\mathcal{F}(N^-),\mathcal{B}_f}(K,A_{f,1})$ by Lemma \ref{lkc} and since $\loc_q(\kappa_f(1,c_i/q))=0$ for any $1 \leq i \leq \nu +1$ and any $q | c_i$ by Lemma \ref{sfi} (in particular the matrix \eqref{mat} is upper-triangular). In order to prove the first statement we just need to prove that they generate $\Sel^{\epsilon_\nu}_{\mathcal{F}(N^-),\mathcal{B}_f}(K,A_{f,1})$. To do that we fix $s \in \Sel^{\epsilon_\nu}_{\mathcal{F}(N^-),\mathcal{B}_f}(K,A_{f,1})$ and by eventually subtracting a linear combination of our classes we can assume that $\loc_{q_i}(s)=0$ for any $i=\nu+1,\dots ,2 \nu +1$. We assume by contradiction that $s \neq 0$. By Lemma \ref{sfi} we have that $\loc_{q_{2\nu +1}}(\kappa_f(1,c_{\nu+1}q_{2\nu+1})) \neq 0$ and so, in particular, $\kappa_f(1,c_{\nu+1}q_{2\nu+1}) \neq 0$. It lies in $\Sel^{-\epsilon_\nu}_{\mathcal{F}(N^-)}(K,A_{f,1})$ and so by Lemma \ref{lli} we can find a Kolyvagin prime $q_{2\nu +2}$ such that
$$\loc_{q_{2\nu+2}}(s) \neq 0,\:\:\:\:\:\:\:\: \loc_{q_{2\nu+2}}(\kappa_f(1,c_{\nu+1}q_{2 \nu+1})) \neq 0.$$
Again by Lemma \ref{sfi} we have that $\loc_{q_{2 \nu +2}}(\kappa_f(1,c_{\nu+1} q_{2\nu+1} q_{2\nu+2})) \neq 0$ and so, in particular, $\kappa_f(1,c_{\nu+1} q_{2\nu+1} q_{2\nu+2}) \neq 0$. By Lemma \ref{lkc} and Lemma \ref{ltp} we have that
\begin{flalign*}
&& 0 &=\langle s,\kappa_f(1,c_{\nu+1} q_{2\nu+1} q_{2\nu+2}) \rangle &\\
&& &= \displaystyle \sum_{q \in \mathcal{B}_f} \langle s,\kappa_f(1,c_{\nu+1} q_{2\nu+1} q_{2\nu+2}) \rangle_q+ \displaystyle \sum_{q|c_{\nu+1} q_{2\nu+1} q_{2\nu+2}} \langle s,\kappa_f(1,c_{\nu+1} q_{2\nu+1} q_{2\nu+2}) \rangle_q.
\end{flalign*}
The first sum must be zero by definition of base locus and $\langle s,\kappa_f(1,c_{\nu+1} q_{2\nu+1} q_{2\nu+2}) \rangle_q=0$ for any $q|c_{\nu +1}q_{2 \nu +1}$. So
$$0=\langle s,\kappa_f(1,c_{\nu+1} q_{2\nu+1} q_{2\nu+2}) \rangle_{q_{2 \nu + 2}} \neq 0$$
where the inequality comes from Lemma \ref{ltp}. It is obviously a contradiction and so the statement is proved.

Now we prove the second statement. We fix a family of classes
$$\{s_i\}_{i=1,\dots ,\nu+1} \in \Sel^{-\epsilon_\nu}_{\mathcal{F}(N^-),\mathcal{B}_f}(K,A_{f,1})$$
and we prove that they are linearly dependent. The columns of the matrix $(\loc_{q_{\nu+i}}(s_j))$ (where $1 \leq i \leq \nu$ and $1 \leq j \leq \nu+1$) must be linearly dependent and so there is a non-trivial linear combination $s$ of our classes such that $\loc_{q_{\nu+i}}(s)=0$ for any $1 \leq i \leq \nu$. We assume by contradiction that $s \neq 0$. By Lemma \ref{lli} we can find a Kolyvagin prime prime $q'$ such that $\loc_{q'}(s) \neq 0$ and $\loc_{q'}(\kappa_f(1,c_{\nu+1})) \neq 0$. By Lemma \ref{sfi} we have that $\loc_{q'}(\kappa_f(1,c_{\nu+1}q')) \neq 0$ and so, in particular, $\kappa_f(1,c_{\nu+1}q') \neq 0$. Finally, we have
\begin{flalign*}
&& 0 &=\langle s,\kappa_f(1,c_{\nu+1} q') \rangle &\\
&& &= \displaystyle \sum_{q \in \mathcal{B}_f} \langle s,\kappa_f(1,c_{\nu+1} q') \rangle_q+ \displaystyle \sum_{q|c_{\nu+1} q'} \langle s,\kappa_f(1,c_{\nu+1} q') \rangle_q
\end{flalign*}
and so, as before, we get
$$0 = \langle s, \kappa_f(1,c_{\nu+1}q')\rangle_{q'} \neq 0.$$
This is a contradiction and so $s=0$, i.e., the fixed family of classes is linearly dependent.
\end{proof}

\begin{proposition} \label{parity}
The $\mathbb{F}$-dimension of $\Sel_{\mathcal{F}(N^-)}(K,A_{f,1})$ is odd.
\end{proposition}

\begin{proof}
We follow the idea of \cite[Theorem 3.15]{LPV}. We prove by induction that the dimension of our Selmer group can not be an even non-negative integer $r$. For $r=0$ we just need to notice that the root number of $f$ over $\mathbb{Q}$ is
$$\epsilon(f/\mathbb{Q})=(-1)^{\nu (N^-) +1}=-1,$$
so $L(f,k/2)=0$ and, in particular, $L(f/K,k/2)=0$. So by \cite[Corollary 3.34]{WAN} we have $\Sel_{\mathcal{F}(N^-)}(K,A_{f,1}) \neq 0$ and we conclude.

Now we prove the inductive step. Assume that the dimension of our Selmer group is $r+2$ with $r$ an even natural number. By Lemma \ref{eadm} we can choose $s_1 \in \Sel_{\mathcal{F}(N^-)}(K,A_{f,1}) \setminus \{0\}$ and an admissible prime $\ell_1$ such that $\loc_{\ell_1}(s_1) \neq 0$. Then by Theorem \ref{lr} and Lemma \ref{kl} there exists $g \in S_k(\Gamma_0(\ell_1N^-))$ such that
$$\dim_\mathbb{F} \Sel_{\mathcal{F}(\ell_1N^-)}(K,A_{g,1})=\dim_\mathbb{F} \Sel_{\mathcal{F}(N^-)}(K,A_{f,1}) -1=r+1.$$
Again by Lemma \ref{eadm} we can choose $s_2 \in \Sel_{\mathcal{F}(\ell_1N^-)}(K,A_{g,1}) \setminus \{0\}$ and an admissible prime $\ell_2$ such that $\loc_{\ell_2}(s_2) \neq 0$. Obviously $g$ satisfies Assumption \ref{Ass1} and so by Theorem \ref{lr} and Lemma \ref{kl} there exists $h \in S_k(\Gamma_0(\ell_2 \ell_1N^-))$ such that
$$\dim_\mathbb{F} \Sel_{\mathcal{F}(\ell_2 \ell_1N^-)}(K,A_{h,1})=\dim_\mathbb{F} \Sel_{\mathcal{F}(\ell_1N^-)}(K,A_{g,1}) -1=r$$
but by induction hypothesis it can not happen since $h$ satisfies Assumption \ref{Ass2}.
\end{proof}

Finally, we can prove Kolyvagin's conjecture in general.

\begin{theorem}
$\kappa_f^{\st} \neq 0$.
\end{theorem}

\begin{proof}
We follow the idea of \cite[Theorem 3.17]{LPV}. We prove the theorem by induction on $r=\dim_\mathbb{F} \Sel_{\mathcal{F}(N^-)}(K,A_{f,1})$. We know by Proposition \ref{parity} that $r$ must be odd and we already proved the case $r=1$ in the previous section. We assume that the theorem is true in case $r$ and we prove it in case $r+2$. We choose $\epsilon \in \{ \pm 1\}$ such that
$$\dim_\mathbb{F}\Sel_{\mathcal{F}(N^-)}^\epsilon (K,A_{f,1}) > \dim_\mathbb{F}\Sel_{\mathcal{F}(N^-)}^{-\epsilon}(K,A_{f,1}).$$
We choose $s_1 \in \Sel^\epsilon_{\mathcal{F}(N^-)}(K,A_{f,1}) \setminus \{0\}$ and by Lemma \ref{eadm} we know that there exists an admissible prime $\ell_1$ such that $\loc_{\ell_1}(s_1) \neq 0$. Then by Theorem \ref{lr} and Lemma \ref{kl} there exists $g \in S_k(\Gamma_0(\ell_1N^-))$ such that
$$\dim_\mathbb{F} \Sel_{\mathcal{F}(\ell_1N^-)}(K,A_{g,1})=\dim_\mathbb{F} \Sel_{\mathcal{F}(N^-)}(K,A_{f,1}) -1=r+1.$$
By Lemma \ref{kl} we also have that
$$\dim_\mathbb{F} \Sel^\epsilon_{\mathcal{F}(\ell_1N^-)}(K,A_{g,1})=\dim_\mathbb{F} \Sel^\epsilon_{\mathcal{F}(N^-)}(K,A_{f,1}) -1.$$
Again, we can choose $s_2 \in \Sel^\epsilon_{\mathcal{F}(\ell_1N^-)}(K,A_{g,1}) \setminus \{0\}$ and an admissible prime $\ell_2$ such that $\loc_{\ell_2}(s_2) \neq 0$. Since $g$ satisfies Assumption \ref{Ass1}, by Theorem \ref{lr} and Lemma \ref{kl} there exists $h \in S_k(\Gamma_0(\ell_2 \ell_1N^-))$ such that
$$\dim_\mathbb{F} \Sel_{\mathcal{F}(\ell_2 \ell_1 N^-)}(K,A_{h,1})=\dim_\mathbb{F} \Sel_{\mathcal{F}(\ell_1N^-)}(K,A_{g,1}) -1=r.$$
We observe that $h$ satisfies Assumption \ref{Ass2} and so by induction we know that $\kappa_h^{\st} \neq 0$. By \cite[Corollary 2.18]{LPV} we have
$$\loc_{\ell_1}(\kappa_f(1,c)) \neq 0 \Leftrightarrow \loc_{\ell_2}(\kappa_h(1,c)) \neq 0$$
for any $c \in \Lambda_{\Kol}(h)$ and so we just need to prove that $\ell_2$ is not a base point for $\kappa_h^{\st}$. Since $\dim_\mathbb{F} \Sel_{\mathcal{F}(\ell_2 \ell_1N^-)}(K,A_{h,1})$ is odd we can distinguish two different case.
\begin{itemize}
\item Let us assume that
$$\dim_\mathbb{F} \Sel^\epsilon_{\mathcal{F}(\ell_2 \ell_1N^-)}(K,A_{h,1}) > \dim_\mathbb{F} \Sel^{- \epsilon}_{\mathcal{F}(\ell_2 \ell_1N^-)}(K,A_{h,1})$$
and so $\epsilon_{\nu_h}=\epsilon$. By Proposition \ref{triangulation} we have
$$\Sel^\epsilon_{\mathcal{F}(\ell_2 \ell_1N^-),\mathcal{B}_h}(K,A_{h,1}) = \Sel^\epsilon_{\mathcal{F}(\ell_2 \ell_1N^-)}(K,A_{h,1}).$$
Since $\loc_{\ell_2}(s_2) \neq 0$ we know by Lemma \ref{kl} that $s_2$ does not belong to the latter and, so, to the former. On the other hand $s_2 \in \Sel^\epsilon_{\mathcal{F}^{\ell_2}(\ell_1N^-)}(K,A_{h,1})$ and so $\ell_2 \notin \mathcal{B}_h$.
\item Let us assume that
$$\dim_\mathbb{F} \Sel^\epsilon_{\mathcal{F}(\ell_2 \ell_1N^-)}(K,A_{h,1}) < \dim_\mathbb{F} \Sel^{- \epsilon}_{\mathcal{F}(\ell_2 \ell_1N^-)}(K,A_{h,1}).$$
In this case we have $\epsilon_{\nu_h}=-\epsilon$ and so
$$\dim_\mathbb{F} \Sel^{- \epsilon}_{\mathcal{F}(\ell_2 \ell_1N^-)}(K,A_{h,1}) =\nu_h +1, \:\:\:\:\:\:\:\: \dim_\mathbb{F} \Sel^\epsilon_{\mathcal{F}(\ell_2 \ell_1N^-),\mathcal{B}_h}(K,A_{h,1}) \leq \nu_h$$
by Proposition \ref{triangulation}. Then
\begin{flalign*}
&& \dim_\mathbb{F} \Sel^\epsilon_{\mathcal{F}(\ell_1N^-)}(K,A_{g,1}) &=\dim_\mathbb{F} \Sel^{-\epsilon}_{\mathcal{F}(\ell_1N^-)}(K,A_{g,1}) &\\
&& &\geq \dim_\mathbb{F} \Sel^{-\epsilon}_{\mathcal{F}(\ell_2 \ell_1N^-)}(K,A_{h,1}) =\nu_h+1
\end{flalign*}
where we used Lemma \ref{kl}. If $\ell_2$ is a base point for $\kappa_h^{\st}$ then we have
$$\nu_h+1 \leq \dim_\mathbb{F} \Sel^\epsilon_{\mathcal{F}(\ell_1N^-)}(K,A_{g,1}) \leq \dim \Sel^\epsilon_{\mathcal{F}(\ell_2 \ell_1N^-),\mathcal{B}_h}(K,A_{h,1}) \leq \nu_h$$
and it is obviously a contradiction.
\end{itemize}
\end{proof}

\section{The $p$-part of the Tamagawa number conjecture} \label{sec:8}

In \cite{LV}, Longo and Vigni give a formulation of the $p$-part of the Tamagawa number conjecture for modular motives of generic rank. We recall briefly the statement and the objects involved, for a more precise and complete exposition we refer to \cite[$\mathsection 2$]{LV}. Let $k \geq 2$ be an even positive integer, let $N \geq 3$ be a positive integer, let $p \nmid 6N$ be a rational prime and let $f \in S_k(\Gamma_0(N))$ be a newform of weight $k$, level $N$ and trivial character with $q$-expansion
$$f(q)=\displaystyle \sum_{n \geq 1}a_n(f)q^n.$$
We denote by $E$ the Hecke field of $f$, i.e., $E:=\mathbb{Q}(\{a_n(f)\})$ and by $\rho_{f,\mathfrak{p}}$ the self-dual twist of the Galois representation attached to $f$ and to a prime $\mathfrak{p}$ of $E$ lying over $p$. We also denote by $V_\mathfrak{p}$ its underlying $E_\mathfrak{p}$-vector space, by $T_\mathfrak{p}$ a $G_\mathbb{Q}$-stable lattice of $V_\mathfrak{p}$ and by $A_\mathfrak{p}:=V_\mathfrak{p}/T_\mathfrak{p}$. Furthermore we set
$$\mathcal{O}_p:= \bigoplus_{\mathfrak{p}|p} \mathcal{O}_\mathfrak{p}, \:\:\:\:\:\: E_p:=\bigoplus_{\mathfrak{p}|p} E_\mathfrak{p}, \:\:\:\:\:\: V_p:=\bigoplus_{\mathfrak{p}|p} V_\mathfrak{p}, \:\:\:\:\:\: T_p:=\bigoplus_{\mathfrak{p}|p} T_\mathfrak{p}, \:\:\:\:\:\: A_p:=\bigoplus_{\mathfrak{p}|p} A_\mathfrak{p}.$$
Finally, we denote by $\mathcal{M}$ the modular motive attached to $f$ (see \cite[$\mathsection 2.2$]{LV}).

The statement of the conjecture involves the following objects:

\begin{itemize}
\item $L^*(\mathcal{M},0)$ is the leading term of the Taylor expansion of $L(\mathcal{M},s)$ at $s=0$ where $L(\mathcal{M},s)$ is the $L$-function of $\mathcal{M}$ over $\mathbb{Q}$ as defined in \cite[Definition 2.21]{LV};
\item $\Reg_\mathscr{B}(\mathcal{M})$ is the determinant of the Gillet--Soul\'e height pairing where $\mathscr{B}$ is a fixed $\mathbb{Q}$-basis of the motivic cohomology $H^1_{\mot}(\mathbb{Q},\mathcal{M})$ (see \cite[$\mathsection 2.6, \mathsection 2.7$]{LV});
\item $\shaf_p^{\BK}(\mathbb{Q}, \mathcal{M}):= \bigoplus_{\mathfrak{p}|p}(\Sel_{\BK}(\mathbb{Q},A_\mathfrak{p})/\Sel_{\BK}(\mathbb{Q},A_\mathfrak{p})_{\divv})$ is the Shafarevich--Tate group of $\mathcal{M}$ over $\mathbb{Q}$;
\item $\gamma_f$ is an element of the Betti realization of $\mathcal{M}$ and $\mathcal{I}_p(\gamma_f)$ is an $\mathcal{O}_p$-ideal attached to it (see \cite[$\mathsection 2.23.3$]{LV});
\item $\Omega_\infty:=\Omega_\infty^{(\gamma_f)}$ is the period of $f$ relative to $\gamma_f$ as defined in \cite[Definition 2.8]{LV};
\item $\Tam^p_q(\mathcal{M}):= \bigoplus_{\mathfrak{p}|p} \Tam_q(\mathbb{Q},A_\mathfrak{p})$ is the $p$-part of the Tamagawa ideal of $\mathcal{M}$ over $\mathbb{Q}$ at any rational prime $q$;
\item $\Tam^p_\infty(\mathcal{M}):=\Fitt_{\mathcal{O}_p}(H^1_{\f}(\mathbb{R},T_p))$ is the $p$-part of the Tamagawa ideal of $\mathcal{M}$ over $\mathbb{Q}$ at $\infty$;
\item $A_{\tilde{\mathscr{B}}} \in \GL_r(E_p)$ (where $r:=\dim_\mathbb{Q} H^1_{\mot}(\mathbb{Q},\mathcal{M})$ is the algebraic rank of $\mathcal{M}$) is a transition matrix defined in \cite[$\mathsection 2.23.2$]{LV};
\item $\Tors_p(\mathcal{M}):=\Fitt_{\mathcal{O}_p}(H^0(\Gal(L/\mathbb{Q}),A_p)^{\vee}) \cdot \Fitt_{\mathcal{O}_p}(H^1(\mathbb{Q},T_p)_{\tors})$ is the $p$-torsion part of $f$ where $L$ is the maximal extension of $\mathbb{Q}$ that is unramified outside $Np$ and $( \lozenge )^{\vee}$ denotes the Pontryagin dual.
\end{itemize}

Furthermore, the conjecture can be stated under the assumption that the following three conjectures are true. See, respectively, \cite[$ \mathsection 2.7$]{LV}, \cite[$\mathsection 2.11$]{LV} and \cite[$ \mathsection 2.14$]{LV} for details.

\begin{conjecture} \label{conj1}
The Gillet--Soul\'e height pairing is non-degenerate.
\end{conjecture}

\begin{conjecture} \label{conj2}
One has
$$\frac{L^*(\mathcal{M},0)}{\Omega_\infty \Reg_\mathscr{B}(\mathcal{M)}} \in E^\times.$$
\end{conjecture}

\begin{conjecture} \label{conj3}
The $p$-adic regulator map
$$\reg_p:H^1_{\mot}(\mathbb{Q},\mathcal{M}) \otimes_E E_p \rightarrow \Sel_{\BK}(\mathbb{Q},V_p)$$
is an isomorphism.
\end{conjecture}

Finally, we can state our conjecture.

\begin{conjecture} \label{pTNCconj}
Assume Conjectures \ref{conj1}, \ref{conj2} and \ref{conj3} hold true. Then the equality of fractional $\mathcal{O}_p$-ideals
$$\left( \frac{(k/2-1)! L^*(\mathcal{M},0)}{ \Omega_\infty \Reg_{\mathscr{B}}(\mathcal{M})} \right)= \frac{\Fitt_{\mathcal{O}_p}(\shaf_p^{\BK}(\mathbb{Q},\mathcal{M})) \mathcal{I}_p( \gamma_f) \prod_{q | pN\infty}\Tam^p_q(\mathcal{M})}{(\det(A_{\tilde{\mathscr{B}}}))^2 \Tors_p(\mathcal{M})}$$
holds.
\end{conjecture}

\subsection{Proof for modular motives of analytic rank zero}

We are going to consider the following objects:
\begin{itemize}
\item $X_{N,k}$ is the Kuga--Sato variety of weight $k$ over the modular curve $X(N)$;
\item $\Heeg_{K,N} \leq \CH_0^{k/2}(X_{N,k}/H_K)$ is the Heegner module of level $N$ attached to an imaginary quadratic field $K$ where all the prime factors of $N$ split and such that its Hilbert class field is $H_K$ (see \cite[$\mathsection 4.1.3$]{LV});
\item $\AJ_{f,K,p}:\Heeg_{K,N}^{\Gal(H_K/K)} \rightarrow H^1(K,T_p)$ is the Abel-Jacobi map attached to $f$ as defined in \cite[(4.8)]{LV};
\item $\reg_{\mathbb{Q},p}^{\intt}$ is the integral variant of the $p$-adic regulator map defined in \cite[$\mathsection 3.4.3$]{LV};
\item $\mathfrak{a}_{f,\Gamma(N)}$ is an ideal of $\mathcal{O}_E$ defined in \cite[$\mathsection 4.3$]{LV}.
\end{itemize}

Tamagawa number conjecture for modular motives of analytic rank zero holds true if we assume the following technical hypotheses.

\begin{assumption} \label{Ass3}~
\begin{itemize}
\item $k<p$;
\item $p \nmid  \N(\mathfrak{a}_{f,\Gamma(N)})$;
\item $T_\mathfrak{p}$ is residually irreducible for any $\mathfrak{p}|p$;
\item $\SL_2(\mathbb{Z}_p) \subset \rho_{f,\mathfrak{p}}(G_\mathbb{Q})$ for any $\mathfrak{p}|p$;
\item there exists $q' \parallel N$ such that $a_{q'}(f)=-(q')^{k/2-1}$;
\item there exists an imaginary quadratic field $K$ such that
\begin{itemize}
\item the primes dividing $Np$ split in $K$;
\item the analytic rank of $f^K$ is one;
\item there exists $\mathfrak{p}|p$ such that $\AJ_{f,K,p}$ is injective on $\Heeg^{\Gal(H_K/K)}_{K,N}$.
\end{itemize}
\end{itemize}
\end{assumption}

\begin{theorem} \label{pTNCzero}
If Assumption \ref{Ass3} is satisfied and the analytic rank of $f$ is zero then Conjecture \ref{pTNCconj} is true.
\end{theorem}

The $p$-part of Tamagawa number conjecture for modular motives of rank zero follows almost directly from Theorem \ref{pTNC}. Indeed, the result of Wan tells us that under Assumption \ref{Ass3} we have
$$v_\varpi \left(\frac{L(f,k/2)}{(2\pi i)^{k/2-1} \Omega_f} \right) =\length_{\mathcal{O}_\mathfrak{p}} \Sel_{\BK}(\mathbb{Q},A_\mathfrak{p}) + \displaystyle \sum_{q|Np} t_q(\mathbb{Q},A_\mathfrak{p}) $$
for any prime $\mathfrak{p}|p$. Then we have
$$\left( \frac{L(f,k/2)}{(2 \pi i)^{k/2-1} \Omega_f} \right)= \Fitt_{\mathcal{O}_\mathfrak{p}}(\Sel_{\BK}(\mathbb{Q},A_\mathfrak{p})) \cdot \displaystyle \prod_{q|Np} \Tam_q(\mathbb{Q},A_\mathfrak{p})$$
for any prime $\mathfrak{p}|p$ and by the proof of \cite[Theorem 4.23]{LV} we have
$$\left( \frac{L^*(\mathcal{M},0)}{\Omega_\infty} \right)=\Fitt_{\mathcal{O}_p}\left( \bigoplus_{\mathfrak{p}|p} \Sel_{\BK}(\mathbb{Q},A_\mathfrak{p})\right) \cdot \displaystyle \prod_{q|Np}\Tam_q^p(\mathcal{M}).$$
Thus, our result follows immediately since:
\begin{itemize}
\item $(k/2-1)! \in \mathcal{O}_p^\times$ since $k<p$;
\item $\Reg_\mathscr{B}(\mathcal{M})=1$ by \cite[Theorem 4.21]{LV};
\item $\shaf_p^{\BK}(\mathbb{Q},\mathcal{M})= \bigoplus_{\mathfrak{p}|p}\Sel_{\BK}(\mathbb{Q},A_\mathfrak{p})$ by \cite[Theorem 4.27]{LV};
\item $\mathcal{I}_p(\gamma_f)=\mathcal{O}_p$ by \cite[(4.33)]{LV};
\item $\Tam^p_\infty (\mathbb{Q},\mathcal{M})=\mathcal{O}_p$ by \cite[(2.74)]{LV};
\item $\det(A_{\tilde{\mathscr{B}}})=1$ by \cite[Theorem 4.21]{LV};
\item $\Tors_p(\mathcal{M})=\mathcal{O}_p$ by \cite[Lemma 3.10]{LV3} and \cite[Proposition 3.21]{LV}.
\end{itemize}

\subsection{Proof for modular motives of analytic rank one}

Tamagawa number conjecture for modular motives of analytic rank one holds true if we assume the following technical hypotheses.

\begin{assumption} \label{Ass4}~
\begin{itemize}
\item $N$ is square-free;
\item $p \nmid D_E \N(\mathfrak{a}_{f, \Gamma(N)}) c_E$ where $D_E$ is the discriminant of $E$ and $c_E:=[\mathcal{O}_E: \mathbb{Z}(\{a_n(f) \})]$;
\item $f$ is $p$-isolated (see \cite[Definition 3.5]{LV});
\item $p>k+1$ and $\left| (\mathbb{F}_p^\times)^{k-1} \right|>5$;
\item $\rho_{f,\mathfrak{p}}(G_\mathbb{Q})$ contains $\SL_2(\mathbb{Z}_p)$ for any $\mathfrak{p}|p$;
\item $\overline{\rho_{f,\mathfrak{p}}}$ is ramified at any prime $q|N$ for any $\mathfrak{p}|p$;
\item $\overline{\rho_{f,\mathfrak{p}}}$ is absolutely irreducible when restricted to the absolute Galois group of $\mathbb{Q}_p \left( \sqrt{p^*} \right)$ for any $\mathfrak{p}|p$ where $p^*=(-1)^{\frac{p-1}{2}}p$;
\item there exists $q' | N$ such that $a_{q'}(f)=-(q')^{k/2-1}$;
\item $\reg_{\mathbb{Q},p}^{\intt}$ is an isomorphism of $\mathcal{O}_p$-modules;
\item there exists an imaginary quadratic field $K$ of discriminant $D_K$ such that
\begin{itemize}
\item the primes dividing $Np$ split in $K$;
\item the analytic rank of $f^K$ is zero;
\item for any prime $\mathfrak{p}|p$ the map $\AJ_{f,K,p}$ is injective on $\Heeg_{K,N}^{\Gal(H_K/K)}$;
\item Conjecture \ref{conj3} holds true for the modular motive attached to $f^K$;
\end{itemize}
\item there exists an imaginary quadratic field $K'$ such that
\begin{itemize}
\item the primes dividing $NpD_K$ split in $K'$;
\item the analytic rank of $(f^K)^{K'}$ is one;
\item there exists a prime $\mathfrak{p}|p$ such that $\AJ_{f^K,K',p}$ is injective on $\Heeg^{\Gal(H_{K'}/K')}_{K',N}$.
\end{itemize}
\end{itemize}
\end{assumption}

\begin{theorem}
If Assumption \ref{Ass4} is satisfied and the analytic rank of $f$ is one then Conjecture \ref{pTNCconj} is true.
\end{theorem}

For the proof in the rank one case we refer to the one given by Longo and Vigni in \cite[$\mathsection 4.7, \mathsection 4.8, \mathsection 4.9$]{LV} for ordinary modular forms since it is essentially the same. The only critical steps are the following:
\begin{itemize}
\item in \cite[Theorem 4.32]{LV} they use Kolyvagin's conjecture for the ordinary modular form $f$ with respect to $K$, anyway we know that under Assumption $\ref{Ass4}$ it holds also when $f$ is non-ordinary (see Theorem \ref{Kconj});
\item in the proof of \cite[Theorem 4.41]{LV} they use the $p$-part of the Tamagawa number conjecture for the ordinary modular form $f^K$, anyway we know that under Assumption $\ref{Ass4}$  it holds also when $f^K$ is non-ordinary (see Theorem \ref{pTNCzero}).
\end{itemize}

\bibliographystyle{amsplain}
\bibliography{Database}

\end{document}